\documentclass[10pt]{amsart}

\usepackage{geometry,graphicx,amssymb,amsmath,amsbsy,eucal,amsfonts,mathrsfs,amscd,bm,tcolorbox}

\usepackage[all]{xy}
\usepackage{tikz}
\usetikzlibrary{arrows,shapes}
\usetikzlibrary{arrows, decorations.markings,fit}
\usetikzlibrary{calc,3d}
\usepackage{tikz-3dplot}

\tikzset{
    >=stealth',
    punkt/.style={
           rectangle,
           rounded corners,
           draw=black, very thick,
           text width=6.5em,
           minimum height=2em,
           text centered},
    pil/.style={
           ->,
           thick,
           shorten <=2pt,
           shorten >=2pt,}
}

\usepackage{subcaption}
\numberwithin{equation}{section}

\allowdisplaybreaks[3]

\newtheorem{theorem}{Theorem}[section]
\newtheorem{lemma}[theorem]{Lemma}

\newtheorem{corollary}[theorem]{Corollary}
\newtheorem{proposition}[theorem]{Proposition}
\theoremstyle{definition}

\theoremstyle{remark}
\newtheorem{remark}[theorem]{Remark}

\newcommand{\p}{{\partial}}

\newcommand{\nab}{\nabla}

\newcommand{\dive}{{\ensuremath\mathop{\mathrm{div}\,}}}

\newcommand{\pol}{\EuScript{P}}
\newcommand{\bpol}{\boldsymbol{\pol}}
\newcommand{\bld}[1]{\boldsymbol{#1}}
\newcommand{\bt}{\bld{t}}

\newcommand{\bI}{\bld{I}}

\newcommand{\bv}{\bld{v}}
\newcommand{\bw}{\bld{w}}

\newcommand{\bn}{\bld{n}}
\newcommand{\bu}{\bld{u}}

\newcommand{\bE}{\bld{E}}
\newcommand{\bW}{\bld{W}}
\newcommand{\bV}{\bld{V}}

\newcommand{\bH}{\bld{H}}

\newcommand{\bL}{\bld{L}}
\newcommand{\bphi}{\bm \phi}

\newcommand{\bX}{{\bm X}}

\newcommand{\bff}{{\bm f}}

\newcommand{\bpsi}{\bm \psi}

\newcommand{\bbR}{\mathbb{R}}

\newcommand{\calT}{\mathcal{T}}

\newcommand{\calN}{\mathcal{N}}

\newcommand{\bPsi}{{\bm \Psi}}
\newcommand{\calE}{\mathcal{E}}

\title[General degree isoparametric SV]{General degree divergence-free finite element methods for the Stokes problem on smooth domains}

\author[R. Durst]{Rebecca Durst}
\address{University of Pittsburgh, Department of Mathematics}
\email{RFD17@pitt.edu}
\author[M. Neilan]{Michael Neilan}
\address{University of Pittsburgh, Department of Mathematics}\thanks{The
second author was supported in part by the NSF, grant DMS-2309425.}
\thanks{All data generated or analysed during this study are included in this article.}
\email{neilan@pitt.edu}

\begin{document}

\maketitle

\begin{abstract}
In this paper, we construct and analyze 
divergence-free finite element methods
for the Stokes problem on smooth domains.
The discrete spaces are based on the Scott-Vogelius
finite element pair of arbitrary polynomial degree
greater than two. By combining the Piola transform
with the classical isoparametric framework,
and with a judicious choice of degrees of freedom,
we prove that the method converges with optimal order
in the energy norm. We also show that the discrete
velocity error converges with optimal order in the $L^2$-norm.
Numerical experiments are presented, which support the theoretical results.
\end{abstract}

\section{Introduction}

\thispagestyle{empty}

Divergence-free methods for the Stokes problem have grown in popularity due to the various advantages they present. This includes pressure-robustness, which allows the errors of the pressure and velocity to be decoupled so that the scheme is well-suited to systems in which the pressure term in the Stokes problem is dominant (i.e. systems with a large pressure gradient or small viscosity). Other advantages include mass-conservation and parameter robustness. Consequently, these methods have become an active area of research (see, e.g., \cite{scott1985norm,guzman2018inf,cockburn2007note,falk2013stokes,john2017divergence}). However, most work on these methods is focused on polyhedral domains. The extension to smooth domains (with optimal-order convergence) is non-trivial and only recently have various approaches been proposed \cite{NeilanOtus21,LiuNeilanMaxim23,LiuNeilanOtus}. 

In this paper, we propose an arbitrary degree, divergence-free, isoparametric finite element scheme based
on the Scott-Vogelius pair \cite{scott1985norm}.
On polygonal domains, this approach approximates the velocity 
with continuous, piecewise polynomials of degree $k$,
and approximates the pressure with discontinuous polynomials
of degree $(k-1)$.  It is well known that the stability of this pair
depends on both the triangulation and the polynomial degree $k$. 
We will work on Clough-Tocher splits which
yield a stable element pair provided $k\ge 2$. This is a commonly used method allowing greater flexibility with respect to polynomial degree.

In our approach, we combine this Scott-Vogelius pair with an isoparametric paradigm. To do so,
we apply $k$-degree polynomial diffeomorphisms to define the curvilinear triangulation
and the finite element spaces.  While this approach is classical for isoparametric elements (see \cite{brenner2008mathematical,scott1973finite}),
its extension to divergence-free methods is non-standard and a direct application
of this approach fails to lead to divergence-free and pressure-robust schemes.
In particular, using classical isoparametric Lagrange finite element spaces
for velocity approximations disrupts the divergence-free and pressure-robust properties of the scheme.
Instead, we employ the divergence-preserving Piola transform in the definition 
of the discrete velocity space. This transform is defined
on the macro (unrefined) triangulation, and we treat the resulting finite element spaces
as macro elements defined on the unrefined triangulation.

The primary challenge in this approach lies in the fact that the Piola transform pollutes the continuity of functions in the Lagrange finite element space. More specifically, when the functions in the discrete velocity space are defined by the Piola transform, 
only normal continuity across interior edges is guaranteed.
 Thus, the resulting space is only $H(div)$-conforming. Nonetheless, the spaces are designed to have weak continuity properties that are  leveraged to ensure consistency and stability so that no additional terms in the bilinear form (e.g., penalty terms) are required in the method. 

Consequently, one of the main contributions of this paper is to design a finite element space
that combines the Lagrange finite element space with the Piola transform
and possesses sufficient weak continuity properties
across interior edges. We achieve such a space via a judicious choice of edge degrees of freedom; specifically, these are taken as the Gauss-Lobatto points of interior edges.
This construction allows us to derive a general estimate
of the jumps of discrete velocity functions across interior edges (cf.~Lemma \ref{jumpBound}).

This work is an extension of \cite{NeilanOtus21} where the lowest-order case $k=2$
was considered. As expected, some of the results in \cite{NeilanOtus21} extend to the general case, such as scaling arguments and inf-sup stability. However, the weak continuity properties of the discrete velocity space
is subtle, and a naive extension of \cite{NeilanOtus21} to arbitrary polynomial
degree does not necessarily lead to an optimal-order convergent method. Another contribution is $L^2$ error estimates. Again, this requires new estimates of the discrete velocity functions
across interior edges.

The organization of the paper is as follows. In the next section, we introduce notation, state the properties of the polynomial diffeomorphism, describe the domain discretization, and introduce the Piola transform. We also establish some necessary preliminary results that 
are later used in the convergence analysis. In Section \ref{locSpace}, we define the local finite element spaces and the degrees of freedom and introduce the global spaces in Section \ref{gloSpace}. Also in Section \ref{gloSpace}, we discuss the weak continuity properties of the function spaces and show that the method is inf-sup stable. In Section \ref{stokesFEM}, we introduce the finite element method and derive optimal-order $H^1$ and $L^2$ error estimates for the velocity and pressure solutions, respectively. Then, in Section \ref{l2est}, we prove optimal-order convergence in $L^2$ for the discrete velocity solution, and in Section \ref{numerics}, we provide numerical experiments to verify our theoretical results. Finally, some auxiliary results are proved in Appendices \ref{Ap:ProofInI} and \ref{Ap:ProofEh}.

\section{Preliminaries}

Let $\Omega \subset \mathbb{R}^2$ be an open, bounded, 
and sufficiently smooth domain with boundary $\partial \Omega$. 
We then construct a mesh following the divergence-free isoparametric method outlined in \cite{NeilanOtus21}.

\subsection{Isoparametric framework}
We begin with a shape regular, affine (simplicial) triangulation $\tilde\calT_h$, with sufficiently small mesh size 
$h=\max_{\tilde{T}\in\tilde\calT_h} {\rm diam}(\tilde{T})$. 
Furthermore, we assume that the boundary vertices lie on $\partial \Omega$, that 
$\tilde{\Omega}_h := {\rm{int}}\big(\cup_{\tilde{T}\in \tilde{\mathcal{T}}_h} \bar{\tilde{T}}\big)$ 
is an $\mathcal{O}(h^2)$ polygonal approximation of $\Omega$, and each $\tilde{T}\in \tilde{\mathcal{T}}_h$ has at most two boundary vertices. 

Next, we let $G: \tilde{\Omega}_h \to \Omega$ be a bijective map between the domain and the mesh with $\|G\|_{W^{1,\infty}(\tilde{\Omega}_h)} \leq C$. Here and throughout the paper, $C$ denotes
a generic positive constant that is independent of any mesh parameter and may take on different values
at each occurrence. We define $G$ such that $G\vert_{\tilde{T}}(x) = x$ at all vertices of $\tilde{T}$. Furthermore, we assume that $G$ is the identity map on interior edges, i.e., edges containing at most one vertex on the boundary.

From here, we define a mesh with curved boundaries following a standard isoparametric framework (see e.g. \cite{brenner2008mathematical,bernardi89,Lenoir86,CiarletRaviart72}). In particular, we define $G_h$ to be the piecewise polynomial nodal interpolant of $G$ of degree $\leq k$ ($k\geq 2$), with $\|G_h\|_{W^{1,\infty}(\tilde{T})} \leq C$ and $\|G_h^{-1}\|_{W^{1,\infty}(\tilde{T})} \leq C$ for all $\tilde{T}\in \tilde{\mathcal{T}}_h$. %, where $DG_h$ denotes the Jacobian of $G_h$. 
Then, the isoparametric triangulation and computational domain are given by
\begin{equation*}
    \mathcal{T}_h := \{G_h(\tilde{T}) : \  \tilde{T}\in\tilde{\mathcal{T}}_h\}, \quad \Omega_h := {\rm{int}}\big(\cup_{T\in\mathcal{T}_h}\bar{T}\big).
\end{equation*}
In particular, $\Omega_h$ is an $O(h^{k+1})$ approximation to $\Omega$.
%\MJN{TODO: Look at literature for more assumptions on possibly higher-order norms of $G_h$ and it's inverse.}
We denote by $\|\cdot\|_{H^m_h(\Omega_h)}$ the piecewise norm with respect to $\calT_h$, i.e.,
\[
\|q\|_{H^m_h(\Omega_h)}^2 = \sum_{T\in \calT_h} \|q\|^2_{H^m(T)}.
\]
We also denote by $\nab_h$ the piecewise gradient operator with respect to $\calT_h$, so
that $\nab_h q|_T = \nab (q|_T)$ for all $T\in \calT_h$.

\subsection{The mappings $F_{\tilde{T}}$ and $F_T$}
To define the finite element spaces, we must first construct mappings to and from the affine and curved triangulations, $\tilde{\mathcal{T}}_h$ and $\mathcal{T}_h$. To do so, we define $\hat{T}$ to be the reference triangle with vertices $(1,0)$, $(0,1)$, and $(0,0)$. Then for each $\tilde T\in \tilde \calT_h$, we
let $F_{\tilde{T}}: \hat{T} \to \tilde{T}$ be an affine bijection with $\vert F_{\tilde{T}}\vert_{W^{1,\infty}(\hat{T})}\leq C h_T$ 
and $\vert F_{\tilde{T}}^{-1}\vert_{W^{1,\infty}(\tilde{T})}\leq C h_{\tilde T}^{-1}$
for $h_{\tilde T} ={ \rm{diam}}(\tilde{T})$. 
Subsequently, we may define $F_T: \hat{T} \to T$ by 
$F_T = G_h \circ F_{\tilde{T}}$. For each $T\in \calT_h$, 
the polynomial diffeomorphism $F_T$ and its inverse
and assumed to satisfy the following estimates:
\begin{equation}\label{eqn:FTBounds}
\begin{split}
    |F_T|_{W^{m,\infty}(\hat T)}&\le C h_T^m\ \ (0\le m\le k),\qquad |F_T^{-1}|_{W^{m,\infty}}\le C h_T^{-m}\ \ (0\le m\le (k+1)),\\
    c_1 h_T^2 &\le \det(D F_T)\le c_2 h_T^2.
    \end{split}
\end{equation}
Here, we have $h_T = {\rm{diam}}(G_h^{-1}(T))$, and $c_1,c_2$ are  generic constants independent of $h_T$. % like $C$. 
Furthermore, we note that, due to the assumptions on $G$, the mappings $F_T$ and $F_{\tilde{T}}$ are oriented so that they match at the vertices of $\hat{T}$. Consequently, the mappings are the same on triangles with three interior edges, so that for all such triangles $T \in \mathcal{T}_h$ we have $T = G_h(\tilde{T}) = \tilde{T}$.

\subsection{The boundary regions of $\Omega$ and $\Omega_h$}
With the isoparametric triangulation established, we 
define $\Omega \Delta \Omega_h = (\Omega \setminus \Omega_h) \cup (\Omega_h \setminus \Omega)$ 
and note it may be shown that (see e.g. \cite[Equation 3.9]{BrennerNeilanSung13} for proof)
\begin{equation}\label{eqn:GeoError}
|\Omega \Delta \Omega_h| \leq C h^{k+1}.
\end{equation}

Next, by the construction of $\Omega \Delta \Omega_h$, we have a bound of
the $H^1$ semi-norm in this boundary region.

\begin{lemma}\label{domainEdges}
Let $\bv \in \bH^2(\Omega)$ be extended into $\mathbb{R}^2$
in a way such that $\|\bv\|_{H^2(\bbR^2)}\le C \|\bv\|_{H^2(\Omega)}$.
Then for $h$ sufficiently small,
\begin{equation*}
\|\nabla \bv \|_{L^2(\Omega \triangle \Omega_h)}\leq C h^{\frac{k+1}{2}} \|\bv\|_{H^2(\Omega)}.
\end{equation*}
\end{lemma}
\begin{proof}
Let $d$ be the signed distance function
of $\Omega$ with the convention $d(x)<0$ for $x\in \Omega$.
For $\delta>0$, define
\[
U_{\delta} := \{x\in \bbR^2:\ d(x)<\delta\},
\]
and note that, because $\p \Omega$ is $C^2$,  there holds $\p U_{\delta} \in C^2$
for $\delta>0$ sufficiently small.
We then set
\begin{align*}
    \mathcal{N}_{\delta} := \{x \in U :   |d(x)| < \delta\}
\end{align*}
to be the tubular region around $\p U_\delta$.
    By \cite[Lemma 4.10]{ElliottRanner13}, there holds for all $w \in H^1(U_\delta)$:
    \begin{equation*}
        \|w\|_{L^2(\calN_\delta)}\le C \delta^{1/2} \|w\|_{H^1(U_\delta)}.
    \end{equation*}

Now set $\delta_h = 2 {\rm dist}\{\p\Omega_h,\p\Omega\}=\mathcal{O}(h^{k+1})$,
so that $\Omega \triangle \Omega_h \subset \calN_{\delta_h}$,
and assume $h$ is sufficiently small such that $\p U_{\delta_h}\in C^2$.
We then have 
\begin{align*}
    \|\nab \bv\|_{L^2(\Omega\triangle \Omega_h)}\le C \|\nab \bv\|_{L^2(\calN_{\delta_h})}\le C \delta_h^{1/2} \|\nab \bv\|_{H^1(U_{\delta_h)}}\le Ch^{\frac{k+1}2}\|\bv\|_{H^2(\Omega)}.
\end{align*}
\end{proof}

\subsection{Clough-Tocher Split}
To guarantee inf-sup stability of the proposed divergence-free method, we introduce
on each element a local triangulation given by the Clough-Tocher split. % where a macroelement is divided into three by connecting the vertices with the barycenter. In our framework, 
Let $\hat{T}^{ct} = \{\hat{K}_i\}_{i=1}^3$ be the Clough-Tocher triangulation of the reference triangle, obtained
by connecting the vertices of $\hat T$ to its barycenter.  We define analogous splits on our affine and curved triangulations via $F_{\tilde{T}}$ and $F_T$
(cf.~Figures \ref{fig:Fancy}--\ref{fig:Fancy2}): 
\begin{equation*}
    \tilde{T}^{ct} = \{F_{\tilde{T}} (\hat{K}) : \ \hat{K} \in \hat{T}^{ct}\}, \quad T^{ct} = \{F_T(\hat{K}) : \ \hat{K} \in \hat{T}^{ct}\}.
\end{equation*}
From \eqref{eqn:FTBounds} and the shape-regularity of $\tilde T_h$, it follows that $| T | \leq C | K |$ for all  $K \in T^{ct}$. 

\begin{remark} 
We note that for  the \textit{macroelement} $T \in \mathcal{T}_h$, only edges containing both vertices on $\partial \Omega_h$ may be curved. However, it may be that interior edges of the local triangulations $K \in T^{ct}$ may indeed be curved as well.
\end{remark}

\subsection{The Piola transform}
The final piece we need to construct the divergence-free method is the Piola transform. 
Given $T\in \calT_h$, we define the matrix $A_T:\hat T\to \bbR^{2\times 2}$ to be
the matrix arising in this transform
\begin{equation}\label{eqn:ATDEF}
A_T(\hat x) := \frac{DF_T(\hat x)}{\det(DF_T(\hat x))}.
\end{equation}
In what follows, the local function spaces on each $T\in \mathcal{T}_h$ will be constructed through $A_T$. %with the aid of this transformation. 
Specifically, given a function $\hat \bv:\hat T\to \bbR^2$, its Piola transform yields
the function $\bv:T\to \bbR^2$ with $\bv= (A_T \hat \bv)\circ F_T^{-1}$ (cf.~Section \ref{locSpace}).
It is well-known that this transform is divergence-preserving and normal-continuity preserving,
and its use in the definition of the spaces given  
below allows us to maintain these properties %(divergence-free, pressure-robust)
of the Scott-Vogelius pair on curved triangulations. We emphasize that this transform
is defined with respect to $T\in \calT_h$, not with respect to the triangles in the Clough-Tocher split.

% More specifically, a function $\bv(x)$ living on $T$ will be defined as $A_T(\hat{x})\hat{\bv}(\hat{x})$, where $\hat{\bv}$ is a corresponding function on $\hat{T}$, as we will show in Section \ref{locSpace}. 

% By constructing finite element spaces in this way, the divergence-free property of the Scott-Vogelius pair on curved Clough-Tocher refinements is preserved. %In fact, as shown in \cite[Lemma 5.2]{NeilanOtus21}, 
% Namely, if our finite element solution to the Stokes velocity $\bu_h$ weakly satisfies the divergence-free condition in the finite element method, then the properties of the Piola transform ensure that $\dive \bu_h \equiv 0$ in $\Omega_h$.

\subsection{Bounds and scaling results} 
The following results give bounds on the matrix $A_T$ and its
inverse. %a relationship between the Jacobians of $F_T$ and $F_{\tilde{T}}$. 
We refer to the appendix of \cite{NeilanOtus21} for a proof for 
the case $k=2$. The arguments given there generalize trivially
for $k\ge 2$, and therefore the proof of the following lemma is omitted.
%%%%%%%%%%%%%%%%%%%%%%%%%%%%
%%%%%%%%%%%%%%%%%%%%%%%%%%%%
%%%%%%%%%%%%%%%%%%%%%%%%%%%%
\begin{lemma}\label{defAT}
For each $T\in \calT_h$, there holds
\begin{equation}\label{eqn:ATBounds}
    |A_T|_{W^{m,\infty}(\hat T)}\le C h_T^{m-1}\ (m\ge 0),\qquad |A_T^{-1}|_{W^{m,\infty}(\hat T)}\le \left\{
\begin{array}{ll}
    C h_T^{1+m} & 0\le m\le k-1,\\
    0 & k\le m.
    \end{array}\right.
\end{equation}
\end{lemma}

Additionally, we will make use of the following scaling results from \cite{bernardi89}.
%%%%%%%%%%%%%%%%%%%%%%%%%%%%
%%%%%%%%%%%%%%%%%%%%%%%%%%%%
%%%%%%%%%%%%%%%%%%%%%%%%%%%%
\begin{lemma}
%[\cite{bernardi89}]
\label{scalingBernardi}
Let $T\in \calT_h$ and $\bw\in \bW^{m,p}(T)$ with $m\ge 0$ and $p\in [1,\infty]$.
Let $\hat \bw\in \bW^{m,p}(\hat T)$ be the image of $\bw$ on $\hat T$
with $\hat \bw(\hat x) = \bw(x)$, $x = F_T(\hat x)$ and set $\hat K = F_T^{-1}(K)$
for each $K\in T^{ct}$.
Then for any $K\in T^{ct}$,
\begin{equation}\label{eqn:Bernardi}
\begin{split}
    |\bw|_{W^{m,p}(K)}&\le C h_T^{2/p-m} \sum_{r=0}^m h_T^{2(m-r)}|\hat \bw|_{W^{r,p}(\hat K)},\\
    |\hat \bw|_{W^{m,p}(\hat K)}&\le C h_T^{m-2/p} \sum_{r=0}^m |\bw|_{W^{r,p}(K)}.
\end{split}
\end{equation}
\end{lemma}

In the results that follow, we let $\bn$ denote the outward unit normal of a given domain (understood from context),
 and set $\bt$ to be the unit tangent vector obtained by rotating $\bn$ 90 degrees counterclockwise.

%%%%%%%%%%%%%%%%%%%%%%%%%%%%%%%%%%%%%%%%%%%%%%%%%%%%%%%%%%%%%%%%%%%%%%%%%%%%%%%%%%
%%%%%%%%%%%%%%%%%%%%%%%%%%%%%%%%%%%%%%%%%%%%%%%%%%%%%%%%%%%%%%%%%%%%%%%%%%%%%%%%%%
%%%%%%%%%%%%%%%%%%%%%%%%%%%%%%%%%%%%%%%%%%%%%%%%%%%%%%%%%%%%%%%%%%%%%%%%%%%%%%%%%%
\section{Local spaces}\label{locSpace}

The full derivation of the local spaces for the divergence-free isoparametric framework can be found in \cite{NeilanOtus21} for 
the case $k=2$. Below, we have the analogous results for general $k$.

To begin, we define the local function space on the reference triangle $\hat{T} \subset \mathbb{R}^2$, with Clough-Tocher triangulation $\hat{T}^{ct} = \{\hat{K}_1, \hat{K}_2, \hat{K}_3\}$. Without including boundary conditions, the polynomial spaces on the reference triangle are
\begin{alignat*}{1}
    \hat{\bV}_k =& \{\hat{\bv}\in \bH^1(\hat{T}) : \hat{\bv}\vert_{\hat{K}} \in {\bpol}_k(\hat{K}) \ \forall\hat{K}\in \hat{T}^{ct}\}, \\
    \hat{Q}_{k-1} =& \{\hat{q}\in L^2(\hat{T}) : \hat{q}\vert_{\hat{K}} \in \pol_{k-1}(\hat{K}) \ \forall \hat{K}\in \hat{T}^{ct}\},
\end{alignat*}
where $\pol_k(S)$ is the space of scalar polynomials of degree $\leq k$ on domain $S$, and $\bpol_k(S) = [\pol_k(S)]^2$.

With $\tilde{x} = F_{\tilde{T}}(\hat{x})$, we define the local spaces on the affine triangle $\tilde{T} \in \tilde{\mathcal{T}}_h$ via composition:
\begin{alignat*}{1}
    \tilde{\bV}_k(\tilde{T}) =& \{\tilde{\bv}\in \bH^1(\tilde{T}) : \tilde{\bv}(\tilde{x}) = \hat{\bv}(\hat{x}), \ \exists \hat{\bv}\in \hat{\bV}_k \}, \\
    \tilde{Q}_{k-1}(\tilde{T}) =& \{\tilde{q}\in L^2(\tilde{T}) : \tilde{q}(\tilde{x}) = \hat{q}(\hat{x}), \ \exists \hat{q}\in \hat{Q}_{k-1} \}.
\end{alignat*}

To incorporate boundary conditions, we further define
\begin{alignat*}{2}
&\hat{\bV}_{k,0} = \hat{\bV}_k \cap \bH_0^1(\hat{T}), \quad  &&\hat{Q}_{k-1,0} = \hat{Q} \cap L_0^2(\hat{T}),\\
&\tilde{\bV}_{k,0}(\tilde{T}) = \tilde{\bV}_k(\tilde{T}) \cap \bH_0^1(\tilde{T}), \quad &&\tilde{Q}_{k-1,0}(\tilde{T}) = \tilde{Q}(\tilde{T}) \cap L_0^2(\tilde{T}),
\end{alignat*}
where $L_0^2(\tilde T)$ is the space of $L^2(\tilde T)$-functions with vanishing mean.

We then define the function spaces on the triangles $T \in \mathcal{T}_h$ (which may have a curved edge), using 
the notation $x = F_{T}(\hat{x})$ and the Piola transform:
\begin{alignat*}{1}
\bV_k(T) =& \{\bv \in \bH^1(T) : \bv(x) = A_T(\hat{x})\hat{\bv}(\hat{x}), \ \exists \hat{\bv}\in \hat{\bV}_k\}, \quad \bV_{k,0}(T) = \bV_k(T) \cap \bH_0^1(T), \\
Q_{k-1}(T) =& \{q \in L^2(T) : q(x) = \hat{q}(\hat{x}), \ \exists \hat{q} \in \hat{Q}_{k-1}\}, \\
Q_{k-1,0}(T) =& \{q \in L^2(T) : q(x) = \hat{q}(\hat{x}), \ \exists \hat{q} \in \hat{Q}_{k-1,0}\},
\end{alignat*}
where the matrix $A_T$ is given by \eqref{eqn:ATDEF}.
Note that if $F_T$ is affine, then $\bV(T) = \tilde{\bV}(\tilde{T})$ and $Q(T) = \tilde{Q}(\tilde{T})$.

It is important to note that functions in $\bV_{k-1}(T)$ and $Q_{k-1}(T)$ are not necessarily piecewise-polynomial spaces if $T$ is not affine. In addition, on curved triangles the matrix $A_T$ is not necessarily constant
on straight edges, and therefore functions in $\bV_{k-1}(T)$ are not necessarily
polynomials on such edges.
However, the following lemma shows that the the normal component of $\bv$ will be a polynomial when restricted to a straight edge. 
%%%%%%%%%%%%%%%%%%%%%%%%%%%%%%%%
%%%%%%%%%%%%%%%%%%%%%%%%%%%%%%%%
%%%%%%%%%%%%%%%%%%%%%%%%%%%%%%%%
\begin{lemma}\label{normsPoly}
Let $\bv \in \bV_k(T)$, and suppose that $e$ is a straight edge of $\partial T$ with unit normal $\bn$. Then $\bv \cdot \bn \vert_e \in \pol_k(e)$.
\end{lemma}
The proof of this result is found in \cite[Lemma 3.1]{NeilanOtus21} for the case $k=2$
and essentially uses the well-known normal-preserving properties
of the Piola transform.
%that $\hat{\bv}\cdot \hat{\bn}$ is a quadratic polynomial. 
However, the result extends trivially to $\hat{\bv} \cdot \hat{\bn}$ a polynomial of arbitrary degree $k$.

With the local spaces now defined, we may state the following lemma showing
that functions in the finite element space $\bV_k(T)$ enjoy
inverse estimates similar to those for piecewise polynomials.
In addition, similar to isoparametric (polynomial) elements
defined via composition, high-order Sobolev norms
of functions in $\bV_k(T)$ are controlled by $k$th-order Sobolev norms.
Its proof is based on scaling arguments and is found in Appendix \ref{Ap:ProofInI}.
%%%%%%%%%%%%%%%%%%%%%%%%%%%%%%%%
%%%%%%%%%%%%%%%%%%%%%%%%%%%%%%%%
%%%%%%%%%%%%%%%%%%%%%%%%%%%%%%%%

\begin{lemma}\label{lem:InvI}
Let $p,q\in [1,\infty]$ and $0\le m\le \ell$ be 
integers.  Then for any $T\in \calT_h$
and $\bv\in \bV_k(T)$,
\begin{align}\label{eqn:InvI1}
\|\bv\|_{W^{\ell,p}(K)}\le C h_T^{m-\ell+2(\frac{1}{p}-\frac1{q})} \|\bv\|_{W^{m,q}(K)}\qquad \forall K\in T^{ct}.
\end{align}
Moreover,
\begin{align}\label{eqn:InvI2}
    \|\bv\|_{W^{\ell,p}(K)}\le  C\|\bv\|_{W^{k,p}(K)}\qquad \forall \ell\ge k,\ \forall K\in T^{ct}.
\end{align}
\end{lemma}

\subsection{Degrees of Freedom on $\bV(T)$} \label{DOFs}
To describe the degrees of freedom
of the local velocity space $\bV(T)$, 
we first summarize the canonical
degrees of freedom for the reference space $\hat \bV_k$, i.e.,
the $k$th-degree Lagrange finite element space defined on the 
Clough-Tocher split.
It is well known that a function in this space
is uniquely determined  by (i) its values at the four vertices in $\hat T^{ct}$ (4 nodes); (ii)
its values at $(k-1)$ distinct points for each of the six (open) edges in $\hat T^{ct}$ ($6(k-1)$ nodes);
and (iii) its values at $\frac12(k-1)(k-2)$ distinct points for each of the three (open) subtriangles
in $\hat T^{ct}$ ($\frac32(k-1)(k-2)$ nodes). In (iii), the $\frac12(k-1)(k-2)$ points for each subtriangle must be chosen such that they uniquely determine a polynomial of degree $(k-3)$.
We see that the total number of nodes is $M_k:=4+6(k-1)+\frac32(k-1)(k-2) = \frac32k(k+1)+1$,
and therefore the dimension of $\hat \bV_k$ is $3k(k+1)+2$.
By setting $\hat{\calN}_k = \{\hat a_i\}_{i=1}^{M_k}$ to be
the set of these points, then a function $\hat \bv\in \hat \bV_k$ is uniquely determined
by the values $\hat \bv(\hat a_i)$ for all $a_i\in \hat{\calN}_k$.

To ensure sufficient weak continuity properties
of the global finite element spaces defined below,
we specify that the location of the points
on the three boundary edges of $\hat T$ correspond 
to the nodes of the Gauss-Lobatto quadrature scheme.
In particular, for a boundary edge $\hat e\subset \p \hat T$,
the nodes on the closure of the edge, denoted by $\{\hat m_i\}_{i=1}^{k+1} \subset \bar{\hat e}$,
satisfy
\[
\sum_{i=1}^{k+1} \hat \omega_i \hat q(\hat m_i) = \int_{\hat e} \hat q\qquad \forall \hat q\in \pol_{2k-1}(\hat e),
\]
and two of the nodes in this set correspond to the vertices of $\hat e$.
The other nodes (i.e., nodes on interior edges and the interior of subtriangles)
can be chosen such that they satisfy the above properties to form a unisolvent set
of degrees of freedom on $\hat \bV_k$. However,
to simplify the implementation of the resulting finite element spaces,
we also take nodes on the interior edges to be the Gauss-Lobttto points,
and let the nodes in the interior of subtriangles to be the canonical
 Lagrange nodes (cf.~Figures \ref{fig:Fancy} and \ref{fig:Fancy2}).

% \MJN{We take the other nodes in $\hat \calN_k$ (i.e., nodes on interior edges
% and the interior of subtriangles) to be the canonical Lagrange nodes, which 
% are equally placed on interior edges and interior subtriangles.} \MJN{[Rephrase? E.g., can
% be any nodes provided they satisfy the conditions in the proceeding paragraph?]}

We map these nodes to  $\tilde T\in \tilde \calT_h$ and $T\in \calT_h$
via the mappings $F_{\tilde T}$ and $F_T$, respectively:
\begin{align*}
    \calN_{k}(\tilde T) = \{F_{\tilde T}(\hat a_i):\ \hat a_i\in \hat \calN_k\},\quad
    \calN_{k}(T) = \{F_{T}(\hat a_i):\ \hat a_i\in \hat \calN_k\}.
\end{align*}
Due to the invariance of spaces of polynomials under affine transformations,
we see that the nodes in $\calN_{k}(\tilde T)$ that lie 
on an edge $\tilde e\subset \bar{\tilde T}$ correspond to the Gauss-Lobatto quadrature
rule on that edge.  Likewise, if $e\subset \p T$ is a straight edge of $T$,
then the nodes in $\calN_k(T)$ that lie on $\bar e$
are the nodes of the Gauss-Lobatto quadrature rule on $e$.

Finally, we note that, since functions in $\hat \bV_k$
are uniquely determined by their values at the nodes $\hat \calN_k$,
and since the matrix $A_T$ is invertible, it follows that any $\bv\in \bV_k(T)$ is uniquely
determined by the values $\bv(a_i)$, $a_i\in \calN_k(T)$.

%%%%%%%%%%%%%%%
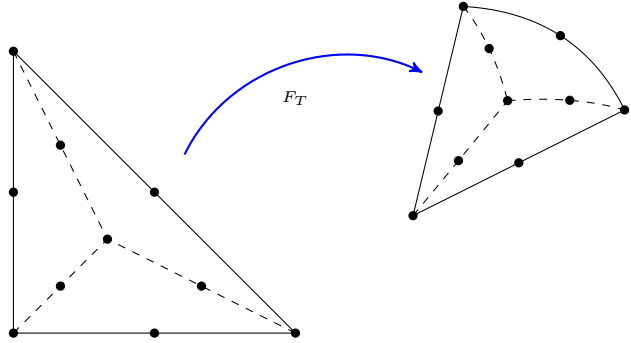
\begin{figure}[h]
\centering
\begin{tikzpicture}[scale=1.25]
\draw[-](0,0)--(3,0);
\draw[-](3,0)--(0,3);
\draw[-](0,3)--(0,0);
\draw[-,dashed](0,0)--(1,1);
\draw[-,dashed](3,0)--(1,1);
\draw[-,dashed](0,3)--(1,1);

%\draw (1.5,-0.3) node {\tiny $\hat a_1$};
%\draw (-0.3,1.45) node {\tiny $\hat a_2$};

%\draw (-0.2,-0.2) node {\tiny $\hat a_3$};

%\draw (3.2,-0.2) node {\tiny $\hat a_8$};
%\draw (1.65,1.65) node {\tiny $\hat a_9$};
%\draw (-0.2,3.2) node {\tiny $\hat a_{10}$};

\node[inner sep = 0pt,minimum size=3.5pt,fill=black!100,circle] (n2) at (0,0)  {};
\node[inner sep = 0pt,minimum size=3.5pt,fill=black!100,circle] (n2) at (3,0)  {};
\node[inner sep = 0pt,minimum size=3.5pt,fill=black!100,circle] (n2) at (0,3)  {};
\node[inner sep = 0pt,minimum size=3.5pt,fill=black!100,circle] (n2) at (1.5,0)  {};
\node[inner sep = 0pt,minimum size=3.5pt,fill=black!100,circle] (n2) at (0,1.5)  {};
\node[inner sep = 0pt,minimum size=3.5pt,fill=black!100,circle] (n2) at (1,1)  {};
\node[inner sep = 0pt,minimum size=3.5pt,fill=black!100,circle] (n2) at (0.5,0.5)  {};
\node[inner sep = 0pt,minimum size=3.5pt,fill=black!100,circle] (n2) at (0.5,2)  {};
\node[inner sep = 0pt,minimum size=3.5pt,fill=black!100,circle] (n2) at (2,0.5)  {};
\node[inner sep = 0pt,minimum size=3.5pt,fill=black!100,circle] (n2) at (1.5,1.5)  {};

%%%%%%%%%%
%% curved triangle
%%%%%%%%%%

%two affine edges
\draw[-](4.785,3.4768875)--(4.25,1.25);
\draw[-](4.25,1.25)--(6.5,2.375);

%mapping
 \node (phys) at (4.5,2.75) {};
 \node (ref)  at (1.75,1.75) {}
 edge[pil,bend left=45,blue] (phys.west);
 \draw (3,2.5) node {\tiny$F_{T}$};
 
 %boundary nodes
 \node[inner sep = 0pt,minimum size=3.5pt,fill=black!100,circle] (n2) at (4.785,3.4768875)  {};
  \node[inner sep = 0pt,minimum size=3.5pt,fill=black!100,circle] (n2) at (4.25,1.25)  {};
  \node[inner sep = 0pt,minimum size=3.5pt,fill=black!100,circle] (n2) at (6.5,2.375)  {};

%node labeling  
%  \draw (5.37500000000000,1.65) node {\tiny $a_1$};
%\draw (4.25,2.4) node {\tiny $a_2$};
%\draw (6,3.4) node {\tiny $a_{9}$};
% \draw  (4.785,3.6768875) node {\tiny $a_{10}$};
%   \draw  (4.05,1.25)  node {\tiny $ a_3$};
%     \draw  (6.7,2.375) node {\tiny $a_8$};

%edge midpoints  
    \node[inner sep = 0pt,minimum size=3.5pt,fill=black!100,circle] (n2) at (4.51750000000000,2.36344375000000)  {};
    \node[inner sep = 0pt,minimum size=3.5pt,fill=black!100,circle] (n2) at (5.37500000000000,1.8125)  {};
  \node[inner sep = 0pt,minimum size=3.5pt,fill=black!100,circle] (n2) at (5.8175,3.16584688)  {};
  
  %curved boundary edge
  \draw[black, domain=0:1,smooth,variable=\t]plot (4.78500000000000 - 0.699999999999999*\t^2 + 2.41500000000000*\t,3.47688750000000 - 0.959612520000000*\t^2 - 0.142274980000000*\t);
  
  %(parameterizations of (curved) CT split
      \draw[black, dashed,domain=0:0.333,smooth,variable=\t]plot (4.25000000000000 + 0.699999999999999*\t^2 + 2.78500000000000*\t,1.25000000000000 + 0.959612520000000*\t^2 + 3.35188750000000*\t);
    \draw[black, dashed,domain=0.333:1,smooth,variable=\t]plot (4.51750000000000 - 0.350000000000000*\t^2 + 2.33250000000000*\t,2.36344375000000 - 0.479806260000000*\t^2 + 0.491362510000000*\t);
    \draw[black, dashed,domain=0:0.333,smooth,variable=\t]plot (4.78500000000000 - 1.40000000000000*\t^2 + 1.88000000000000*\t,3.47688750000000 - 1.91922504000000*\t^2 - 2.36916248000000*\t);    
    
%interior nodes
 \node[inner sep = 0pt,minimum size=3.5pt,fill=black!100,circle] (n2) at (5.25611111111111,2.47391944666667)  {};
 \node[inner sep = 0pt,minimum size=3.5pt,fill=black!100,circle] (n2) at (4.73361111111111,1.83530382000000)  {};
 \node[inner sep = 0pt,minimum size=3.5pt,fill=black!100,circle] (n2) at (5.91694444444444,2.47777153000000)  {}; 
  \node[inner sep = 0pt,minimum size=3.5pt,fill=black!100,circle] (n2) at (5.05944444444444,3.02871528000000)  {}; 
    
\end{tikzpicture}
\caption{\label{fig:Fancy}Clough-Tocher split and degrees of freedom 
for the quadratic Lagrange finite element space $(k=2$).  The local split
is mapping to the curved element $T$ via the polynomial diffeomorphism $F_T$.}
\end{figure}

%%%%%%%%%%%%%%%%%%%%%%%%%%%%%%%%%%%%

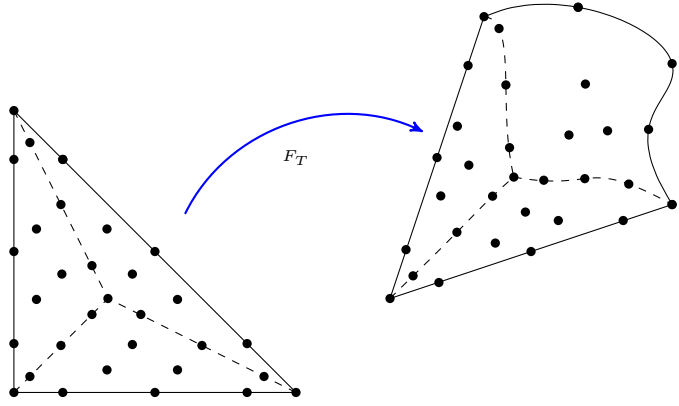
\begin{figure}[h]
\centering
\begin{tikzpicture}[scale=1.25]
\draw[-](0,0)--(3,0);
\draw[-](3,0)--(0,3);
\draw[-](0,3)--(0,0);

\draw[-,dashed](0,0)--(1,1);
\draw[-,dashed](3,0)--(1,1);
\draw[-,dashed](0,3)--(1,1);

%\draw (1.5,-0.3) node {\tiny $\hat a_1$};
%\draw (-0.3,1.45) node {\tiny $\hat a_2$};

%\draw (-0.2,-0.2) node {\tiny $\hat a_3$};
%
%
%\draw (3.2,-0.2) node {\tiny $\hat a_8$};
%\draw (1.65,1.65) node {\tiny $\hat a_9$};
%\draw (-0.2,3.2) node {\tiny $\hat a_{10}$};

\node[inner sep = 0pt,minimum size=3.5pt,fill=black!100,circle] (n2) at (0,0)  {};

\node[inner sep = 0pt,minimum size=3.5pt,fill=black!100,circle] (n2) at (3,0)  {};

\node[inner sep = 0pt,minimum size=3.5pt,fill=black!100,circle] (n2) at (0,3)  {};
\node[inner sep = 0pt,minimum size=3.5pt,fill=black!100,circle] (n2) at (1.5,0)  {};
\node[inner sep = 0pt,minimum size=3.5pt,fill=black!100,circle] (n2) at (0,1.5)  {};
\node[inner sep = 0pt,minimum size=3.5pt,fill=black!100,circle] (n2) at (1,1)  {};

\node[inner sep = 0pt,minimum size=3.5pt,fill=black!100,circle] (n2) at (0.5,0.5)  {};

\node[inner sep = 0pt,minimum size=3.5pt,fill=black!100,circle] (n2) at (0.83,0.83)  {};

\node[inner sep = 0pt,minimum size=3.5pt,fill=black!100,circle] (n2) at (0.17,0.17)  {};

\node[inner sep = 0pt,minimum size=3.5pt,fill=black!100,circle] (n2) at (0.5,2)  {};

\node[inner sep = 0pt,minimum size=3.5pt,fill=black!100,circle] (n2) at (0.83,1.35)  {};

\node[inner sep = 0pt,minimum size=3.5pt,fill=black!100,circle] (n2) at (0.17,2.66)  {};

\node[inner sep = 0pt,minimum size=3.5pt,fill=black!100,circle] (n2) at (2,0.5)  {};

\node[inner sep = 0pt,minimum size=3.5pt,fill=black!100,circle] (n2) at (2.66,0.17)  {};

\node[inner sep = 0pt,minimum size=3.5pt,fill=black!100,circle] (n2) at (1.35,0.83)  {};

\node[inner sep = 0pt,minimum size=3.5pt,fill=black!100,circle] (n2) at (1.5,1.5)  {};

\node[inner sep = 0pt,minimum size=3.5pt,fill=black!100,circle] (n2) at (0.52,2.48)  {};

\node[inner sep = 0pt,minimum size=3.5pt,fill=black!100,circle] (n2) at (0.52,2.48)  {};

\node[inner sep = 0pt,minimum size=3.5pt,fill=black!100,circle] (n2) at (2.48,0.52)  {};

\node[inner sep = 0pt,minimum size=3.5pt,fill=black!100,circle] (n2) at (0,0.52)  {};

\node[inner sep = 0pt,minimum size=3.5pt,fill=black!100,circle] (n2) at (0,2.48)  {};

\node[inner sep = 0pt,minimum size=3.5pt,fill=black!100,circle] (n2) at (2.48,0)  {};

\node[inner sep = 0pt,minimum size=3.5pt,fill=black!100,circle] (n2) at (0.52,0)  {};

%%%%%5
%Interior
%%%%%
\node[inner sep = 0pt,minimum size=3.5pt,fill=black!100,circle] (n2) at (0.99,0.24)  {};

\node[inner sep = 0pt,minimum size=3.5pt,fill=black!100,circle] (n2) at (1.26,0.51)  {};

\node[inner sep = 0pt,minimum size=3.5pt,fill=black!100,circle] (n2) at (1.74,0.24)  {};

\node[inner sep = 0pt,minimum size=3.5pt,fill=black!100,circle] (n2) at (0.24,0.99)  {};

\node[inner sep = 0pt,minimum size=3.5pt,fill=black!100,circle] (n2) at (0.51,1.26)  {};

\node[inner sep = 0pt,minimum size=3.5pt,fill=black!100,circle] (n2) at (0.24,1.74)  {};

\node[inner sep = 0pt,minimum size=3.5pt,fill=black!100,circle] (n2) at (1.26,1.26)  {};

\node[inner sep = 0pt,minimum size=3.5pt,fill=black!100,circle] (n2) at (0.99,1.74)  {};

\node[inner sep = 0pt,minimum size=3.5pt,fill=black!100,circle] (n2) at (1.74,0.99)  {};

%%%%%%%%%%
%% curved triangle
%%%%%%%%%%
%affine edges
\draw[-](5,4)--(4,1);
\draw[-](4,1)--(7,2);

 %curved boundary edge
  %\draw[black, domain=0:1,smooth,variable=\t]plot (5.00140000000000 + 21.1329*\t^4 - 33.6474000000000*\t^3 + 9.48960000000000*\t^2 + 5.02349999999999*\t,4.00000000000000 - 10.6920*\t^4 + 21.3111000000000*\t^3 - 15.2532000000000*\t^2 + 2.63550000000000*\t);
  
%curved boundary edge
  \draw[black, domain=0:1,smooth,variable=\t]plot (12.9681*\t^4 - 19.9989*\t^3 + 3.3075*\t^2 + 5.7219*\t + 5.0014,4.00000000000000 - 10.6920*\t^4 + 21.3111000000000*\t^3 - 15.2532000000000*\t^2 + 2.63550000000000*\t);

  %curved boundary nodes
 \node[inner sep = 0pt,minimum size=3.5pt,fill=black!100,circle] (n2) at (5,4)  {};

  \node[inner sep = 0pt,minimum size=3.5pt,fill=black!100,circle] (n2) at (7,2)  {};

   \node[inner sep = 0pt,minimum size=3.5pt,fill=black!100,circle] (n2) at (6.75,2.8)  {};

    \node[inner sep = 0pt,minimum size=3.5pt,fill=black!100,circle] (n2) at (7,3.5)  {};

     \node[inner sep = 0pt,minimum size=3.5pt,fill=black!100,circle] (n2) at (6,4.1)  {};

    % %(parameterizations of (curved) CT split I GOT ITTTTTT!!!!!!!!
       \draw[black, dashed,domain=0:1,smooth,variable=\t]plot (5.00140000000000 - 0.7728*\t^4 + 2.64580000000000*\t^3 - 2.90950000000000*\t^2 + 1.35180000000000*\t,4.00000000000000 - 1.0444*\t^4 + 3.72360000000000*\t^3 - 4.26550000000000*\t^2 - 0.121500000000000*\t); 

       \draw[black, dashed,domain=0:1,smooth,variable=\t]plot (5.31670000000000 + 1.7910*\t^4 - 1.47730000000000*\t^3 - 0.571000000000001*\t^2 + 1.94060000000000*\t,2.29220000000000 - 1.0762*\t^4 + 0.477700000000000*\t^3 + 0.651500000000000*\t^2 - 0.343800000000000*\t); 

       \draw[black, dashed,domain=0:1,smooth,variable=\t]plot (3.99999999999999 + 1.3299*\t^4 - 2.41010000000000*\t^3 + 1.07540000000000*\t^2 + 1.32150000000000*\t,1.00000000000000 + 0.9326*\t^4 - 1.81260000000000*\t^3 + 0.850699999999999*\t^2 + 1.32150000000000*\t);

 %boundary nodes
 \node[inner sep = 0pt,minimum size=3.5pt,fill=black!100,circle] (n2) at (4,1)  {};

  %boundary nodes
 \node[inner sep = 0pt,minimum size=3.5pt,fill=black!100,circle] (n2) at (4.52,1.17)  {};

  %boundary nodes
 \node[inner sep = 0pt,minimum size=3.5pt,fill=black!100,circle] (n2) at (5.5,1.5)  {};

  %boundary nodes
 \node[inner sep = 0pt,minimum size=3.5pt,fill=black!100,circle] (n2) at (6.48,1.83)  {};

  %boundary nodes
 \node[inner sep = 0pt,minimum size=3.5pt,fill=black!100,circle] (n2) at (7,2)  {};

  %boundary nodes
 \node[inner sep = 0pt,minimum size=3.5pt,fill=black!100,circle] (n2) at (4.83,3.48)  {};

  %boundary nodes
 \node[inner sep = 0pt,minimum size=3.5pt,fill=black!100,circle] (n2) at (4.5,2.5)  {};
  %boundary nodes
 \node[inner sep = 0pt,minimum size=3.5pt,fill=black!100,circle] (n2) at (4.17,1.52)  {};
 
%Interior Nodes

\node[inner sep = 0pt,minimum size=3.5pt,fill=black!100,circle] (n2) at (5.12,1.59)  {};

\node[inner sep = 0pt,minimum size=3.5pt,fill=black!100,circle] (n2) at (5.44,1.92)  {};

\node[inner sep = 0pt,minimum size=3.5pt,fill=black!100,circle] (n2) at (5.79,1.83)  {};

\node[inner sep = 0pt,minimum size=3.5pt,fill=black!100,circle] (n2) at (4.539936160, 2.091670599)  {};

\node[inner sep = 0pt,minimum size=3.5pt,fill=black!100,circle] (n2) at (4.838442980, 2.419232387)  {};

\node[inner sep = 0pt,minimum size=3.5pt,fill=black!100,circle] (n2) at (4.716222093, 2.834167208)  {};

\node[inner sep = 0pt,minimum size=3.5pt,fill=black!100,circle] (n2) at (5.903266890, 2.740373611)  {};

\node[inner sep = 0pt,minimum size=3.5pt,fill=black!100,circle] (n2) at (6.078077875, 3.282501084)  {};

\node[inner sep = 0pt,minimum size=3.5pt,fill=black!100,circle] (n2) at (6.313783890, 2.785097663)  {};

%CT interior

      \node[inner sep = 0pt,minimum size=3.5pt,fill=black!100,circle] (n2) at (4.71145625, 1.70513750)  {};

\node[inner sep = 0pt,minimum size=3.5pt,fill=black!100,circle] (n2) at (5.090772422, 2.089067416)  {};

\node[inner sep = 0pt,minimum size=3.5pt,fill=black!100,circle] (n2) at (4.245003984, 1.241113843)  {}; 

\node[inner sep = 0pt,minimum size=3.5pt,fill=black!100,circle] (n2) at (5.635351443, 2.258856901)  {};

\node[inner sep = 0pt,minimum size=3.5pt,fill=black!100,circle] (n2) at (5.3167, 2.2922)  {};

\node[inner sep = 0pt,minimum size=3.5pt,fill=black!100,circle] (n2) at (6.0715250, 2.2756250)  {};

\node[inner sep = 0pt,minimum size=3.5pt,fill=black!100,circle] (n2) at (6.539313694, 2.218060549)  {};

\node[inner sep = 0pt,minimum size=3.5pt,fill=black!100,circle] (n2) at (6.001100078, 4.099876640)  {};

\node[inner sep = 0pt,minimum size=3.5pt,fill=black!100,circle] (n2) at (5.271425202, 2.605571366)  {};

\node[inner sep = 0pt,minimum size=3.5pt,fill=black!100,circle] (n2) at (5.2323500, 3.2730500)  {};

\node[inner sep = 0pt,minimum size=3.5pt,fill=black!100,circle] (n2) at (5.159474815, 3.873493802)  {};

%mapping
 \node (phys) at (4.5,2.75) {};
 \node (ref)  at (1.75,1.75) {}
 edge[pil,bend left=45,blue] (phys.west);
 \draw (3,2.5) node {\tiny$F_{T}$};

\end{tikzpicture}
\caption{\label{fig:Fancy2}Clough-Tocher split and degrees of freedom for 
the quartic Lagrange finite element space ($k=4$). Edge degrees of freedom on $\hat T$ are placed
at Gauss-Lobatto points.}
\end{figure}

\section{Global Spaces}\label{gloSpace}
%\MJN{I will introduce global spaces $\bV_k^h\times Q_{k-1}^h$ here.  I need argue that $\bV_k^h$ is $H(div)$-conforming.}

On the affine triangulation $\tilde{\mathcal{T}}_h$, we define the Scott-Vogelius pair
\begin{alignat*}{1}
\tilde{\bV}_k^{h} =& \{\tilde{\bv} \in \bH_0^{1}(\tilde{\Omega}_h) : \tilde{\bv}\vert_{\tilde{T}}\in \tilde{\bV}_k(\tilde{T}) \ \forall \tilde{T} \in \tilde{\mathcal{T}}_h\}, \\
\tilde{Q}_{k-1}^h =& \{\tilde{q} \in L_0^2(\tilde{\Omega}_h) : \tilde{q}\vert_{\tilde{T}} \in \tilde{Q}_{k-1}(\tilde{T}) \ \forall \tilde{T} \in \tilde{\mathcal{T}}_h\}.
\end{alignat*}
We see that $\tilde{\bV}_k^{h}$ is the $k$th degree Lagrange finite element space
with respect to the Clough-Tocher refinement of $\tilde \calT_h$, and $\tilde{Q}_{k-1}^h$ is 
the space of discontinuous polynomials of degree $(k-1)$, 
again with respect to the Clough-Tocher refinement.
The finite elements $\tilde \bV_h\times \tilde Q_h$ represents
a stable and divergence--free Stokes pair \cite{ArnoldQin92}, however its use formally leads
to a suboptimal scheme on smooth domains due to geometric error.

To define the isoparametric spaces, we define the operators $\bPsi_k$ and $\Upsilon_{k-1}$ such that 
$\bPsi_{k}\vert_{T}:\tilde{\bV}(\tilde{T}) \to \bV(T)$ and $\Upsilon_{k-1}\vert_{T}: \tilde{Q}(\tilde{T}) \to Q(T)$
and are uniquely determined on each $\tilde{T} \in \tilde\calT_h$ and $T \in \calT_h$ with $T = G_h(\tilde{T})$ as follows:
\begin{enumerate}
    \item[1.]  $(\bPsi_{k}\vert_{T}\tilde{\bv})(a) = \tilde{\bv}(\tilde{a})$ $\forall \tilde{a} \in \mathcal{N}_k(\tilde{T})$, with $a = G_h(\tilde{a})\in \calN_k(T)$, and

    \item[2.] $(\Upsilon_{k-1}\vert_{T}\tilde{q}) = \tilde{q}\circ G_h^{-1}$
\end{enumerate}
Thus, $\bPsi_k$ maps functions in $\tilde \bV_k^h$ to the isoparametric
domain $\Omega_h$ via the Piola transform and interpolation, and 
$\Upsilon_{k-1}$ maps functions in $\tilde Q_{k-1}^h$ to $\Omega_h$ via composition.

%We note that  are uniquely defined operators connecting the local function spaces. Additionally, 
In the following proposition, we state some properties
of the mapping $\bPsi_k$ without proof, as the result is proven for $k=2$ in \cite[Theorem 3.7]{NeilanOtus21}. 
To extend the results to arbitrary degree $k$,  one only needs to recognize 
that a $k$-th degree polynomial along an edge $e \subset \partial T$ is uniquely determined 
by its values at the $k+1$ nodal points that lie on this edge. 
%%%%%%%%%%%%%%%%%%%%%%%%%%%%
\begin{proposition}\label{psiUps}
The following properties are satisfied:
\begin{enumerate}
\item[1.] If $F_T$ is affine, then $\bPsi_{k}\vert_{T}$ is the identity operator.
\item[2.] If $e \subset \partial T$ is a straight edge (so $e \subset \partial \tilde{T}$ with $T = G_h(\tilde T)$), then 
\begin{equation*}
(\bPsi_{k}\vert_{T}\tilde{\bv})\cdot \bn \vert_e = \tilde{\bv}\cdot \bn \vert_e\qquad \forall \tilde \bv\in \tilde \bV_k^h.
\end{equation*}
\item[3.] There holds $\|\bPsi_{k}\vert_{T}\tilde{\bv}\|_{H^1(T)} \leq C \|\tilde{\bv}\|_{H^1(\tilde{T})}$ for all $\tilde \bv\in \tilde \bV_k^h$.
\end{enumerate}
\end{proposition}

Consequently, global function spaces  defined on the isoparametric mesh $\mathcal{T}_h$ are given by
\begin{equation*}
    \bV_k^h := \{ \bv : \bv= \bPsi_k \tilde{\bv}, \ \exists \tilde{\bv}\in \tilde{\bV}_k^h\}, \quad Q_k^h := \{q : q = \Upsilon_{k-1} \tilde{q}, \ \exists \tilde{q}\in \tilde{Q}_k^h\}.
\end{equation*}

\begin{remark}
    From the boundary conditions applied to the space $\tilde{\bV}_k^h$ and the definition of $\bPsi_k$, we see 
    that functions in $\bV_k^h$ are continuous at the degrees of freedom, and vanish on $\partial \Omega_h$.
\end{remark}

With these spaces defined, we have the following results.
%%%%%%%%%%%%%%%%%%%%%%%%%%%%%%%%
%%%%%%%%%%%%%%%%%%%%%%%%%%%%%%%
\begin{lemma}\label{Hdiv}
There holds $\bV_k^h \subset \bH_0({\rm{div}};\Omega_h) = \{\bv \in \bL^2(\Omega_h) : \ \dive{\bv}\in L^2(\Omega_h), \ \bv \cdot \bn\vert_{\partial \Omega_h} = 0\}$.
\end{lemma}

This result follows immediately from the construction of $\bV_k^h$ and the continuity of the normal component on interior edges imposed by part 2 of Proposition \ref{psiUps}. 
See \cite[Theorem 4.2]{NeilanOtus21} for details.
%%%%%%%%%%%%%%%%%%%%%%%%%%%%%%%%
%%%%%%%%%%%%%%%%%%%%%%%%%%%%%%%%
%%%%%%%%%%%%%%%%%%%%%%%%%%%%%%%%
\begin{lemma}\label{interp}
There exists an operator $\bI_k^h:\bH^2(\Omega)\cap \bH^1_0(\Omega_h)\to \bV_k^h$
such that for $\bu\in \bH^s(\Omega_h)\cap \bH^1_0(\Omega_h)\ (s\ge 2)$ and for each $T\in \calT_h$, 
there holds
\begin{equation}\label{eqn:interpBounds}
    \|\bu - \bI_k^h \bu\|_{H^m(T)}\leq C h_T^{\ell-m}\|\bu \|_{H^\ell(T)} \quad 0\le m\le \ell:=\min\{k+1,s\}.
\end{equation}
\end{lemma}
\begin{proof}
Recall $\mathcal{N}_k(T)$ and $\hat\calN_{k}$ 
are the sets of nodes on $T$ and $\hat{T}$, respectively. %This lemma and its proof follow closely with the result of Lemma 3.5 in \cite{NeilanOtus21}. 
We uniquely define the operator $\bI_k^h$ such that on each $T\in \calT_h$, 
 \begin{equation*}
 (\bI_k^h \bu)|_T(a) = \bu(a) \quad \forall a \in \mathcal{N}_k(T).
 \end{equation*}
Set $\bv = \bI_k^h \bu|_T \in \bV_k(T)$, for $T\in\mathcal{T}_h$. Then set $\hat \bv\in \hat \bV_k$ and $\hat \bu\in \bH^s(\hat T)$ such that
    \begin{equation*}
        \bv(x) = (A_T \hat{\bv})(\hat{x}), \quad \bu(x) = (A_T \hat{\bu})(\hat{x}).
    \end{equation*}
    Consequently,
    \begin{equation*}
        (A_T\hat{\bv})(\hat{a}) = (A_T\hat{\bu})(\hat{a}) %\quad \forall \hat{a} \in \hat\calN_k
        \qquad \Longrightarrow\qquad \hat{\bv}(\hat{a}) = \hat{\bu}(\hat{a}) \quad \forall \hat{a} \in \hat\calN_k
    \end{equation*}
because the matrix $A_T$ is invertible.
% The matrix $A_T$ is invertible, therefore this yields
% \begin{equation*}
%    ,
% \end{equation*}
Thus, $\hat{\bv}$ is the $k$th degree nodal interpolant of $\hat{\bu}$ with respect to $\hat{T}^{ct}$,
so, by standard interpolation theory, we have
\begin{equation}\label{interpBound}
    \|\hat{\bu} - \hat{\bv}\|_{H^m(\hat{T})} \leq C\vert \hat{\bu}\vert_{H^\ell(\hat{T})}\quad 0\le m\le \ell=\min\{k+1,s\}.
\end{equation}

Thus it follows from \eqref{interpBound}, Lemmas \ref{defAT} and \ref{scalingBernardi}, and an application of the product rule that
\begin{alignat*}{1}
\vert \bu - \bv \vert_{H^{m}(T)} 
%\leq& C h_T^{1-m} \|A_T(\hat{\bu} - \hat{\bv})\|_{H^m(\hat{T})} \\
%\leq&C h_T^{1-m} \sum_{j=0}^{m} \vert A_T \vert_{W^{j,\infty}(\hat{T})} \|\hat{\bu}-\hat{\bv}\|_{H^m(\hat{T})} \\
\leq&C h_T^{1-m}  \| A_T \|_{W^{m,\infty}(\hat{T})} \|\hat{\bu}-\hat{\bv}\|_{H^m(\hat{T})} 
%\leq& C h_T^{1-m} \sum_{j=0}^{m} h^{j-1} \|\hat{\bu}-\hat{\bv}\|_{H^m(\hat{T})} \\
%\leq& C h_T^{1-m} h^{-1} \|\hat{\bu}-\hat{\bv}\|_{H^m(\hat{T})} \\
\leq C h_T^{-m}\vert \hat{\bu}\vert_{H^\ell(\hat{T})},
\end{alignat*}
and so, by using Lemmas \ref{defAT} and \ref{scalingBernardi} once again,
\begin{align*}
\vert \bu - \bv \vert_{H^{m}(T)} 
\le & C h_T^{-m}\vert A_T^{-1}A_T\hat{\bu}\vert_{H^\ell(\hat{T})} \\
\leq& C h_T^{-m} \sum_{j=0}^\ell \vert A_T^{-1} \vert_{W^{j,\infty}(\hat{T})} \vert A_T\hat{\bu}\vert_{H^{\ell-j}(\hat{T})}  \\
%
%=& C h_T^{-m}\vert A_T^{-1}A_T\hat{\bu}\vert_{H^\ell(\hat{T})} \\
\leq& C h_T^{-m} \sum_{j=0}^\ell h_T^{1+j} \vert A_T\hat{\bu}\vert_{H^{\ell-j}(\hat{T})}  \\
%
% =& C h_T^{-m}\vert A_T^{-1}A_T\hat{\bu}\vert_{H^\ell(\hat{T})} \\
% \leq& C h_T^{-m} \sum_{j=0}^\ell h_T^{1+j} \cdot h_T^{\ell-j-1} \|\bu\|_{H^{\ell-j}(T)}\\
% %
% \leq& C h_T^{-m}\bigg(h_T\vert A_T\hat{\bu}\vert_{H^\ell(\hat{T})} + h_T^2 \vert A_T\hat{\bu}\vert_{H^{\ell-1}(\hat{T})}\bigg) \\
% %
% \leq& C h_T^{-m}\bigg(h_T h_T^{\ell-1}\|\bu \|_{H^\ell(T)} + h_T^2 h_T^{\ell-2} \|\bu\|_{H^{\ell-1}(T)}\bigg) \\
\leq& C h_T^{\ell-m}\|\bu\|_{H^\ell(T)}.
\end{align*}
\end{proof}

\subsection{Weak continuity properties}\label{weakContinuity}

The next result shows
that, while functions in $\bV^h_k$ are only
$\bH_0({\rm div};\Omega_h)$-conforming,
they do have weak continuity properties
across interior edges of the mesh. 
In particular, they are ``close'' to an $\bH^1_0(\Omega_h)$-conforming relative.  The lemma
is a generalization of \cite[Lemma 4.5]{NeilanOtus21}
to general polynomial degree
and to higher-order Sobolev norms;
its proof is given in Appendix \ref{Ap:ProofEh}.
%%%%%%%%%%%%%%%%%%%%%%%%%%%%%%%%
%%%%%%%%%%%%%%%%%%%%%%%%%%%%%%%%
%%%%%%%%%%%%%%%%%%%%%%%%%%%%%%%%
\begin{lemma}\label{Eh}
There exists an operator $\bE_h : \bV^h_k \to \bH_0^1(\Omega_h)$ such that for all $\bv \in \bV^h_k$
\begin{equation}\label{EhBound}
\|\bv - \bE_h\bv\|_{L^2(T)} + h_T \|\nabla(\bv - \bE_h \bv)\|_{L^2(T)}\leq C h_T^{m+1}\|\bv\|_{H^m(T)}\qquad \forall T\in \calT_h,
\end{equation}
for $m=0,1,\ldots,k$. Moreover, $\bE_h \bv|_T = \bv$
if $T$ is affine.
\end{lemma}

%%%%%%%%%%%%%%%%%%%%%%%%%%%%%%%%
%%%%%%%%%%%%%%%%%%%%%%%%%%%%%%%%
%%%%%%%%%%%%%%%%%%%%%%%%%%%%%%%%
\begin{corollary}\label{poincare}
% \textcolor{red}{I am keeping this in because I don't think it is trivial, but I have shortened it significantly}
For $\bv \in \bV^h$, it holds
\begin{equation*}
\|\bv\|_{L^2(\Omega_h)} \leq C\|\nabla_h \bv\|_{L^2(\Omega_h)},
\end{equation*}
where $\nab_h$ denotes the piecewise gradient operator
with respect to $\calT_h$, and $C>0$ is a constant depending only on the size of $\Omega_h$
and the shape regularity of $\calT_h$.
\end{corollary}
\begin{proof}
%  We write
% \begin{alignat*}{1}
% \|\bv\|_{L^2(\Omega_h)}^2 =& \sum_{T\in \mathcal{T}_h} \|\bv\|_{L^2(T)}^2 \\
% \leq& 2\sum_{T \in \mathcal{T}_h} \|\bv - \bE_h \bv\|_{L^2(T)}^2 + \|\bE_h \bv\|_{L^2(\Omega_h)}^2.
% \end{alignat*}

Recall that $\bv = \bE_h \bv$ on affine triangles, so $\|\bv - \bE_h\|_{L^2(T)}$ may only be
nonzero on curved $T\in\mathcal{T}_h$, all of which will have at least two vertices on the boundary. We  denote the set of triangles with two boundary vertices as $\mathcal{T}_h^{\partial}$ so that
$\bv|_T= \bE_h \bv|_T$ for $T\in \calT_h\backslash \calT_h^\partial$.
Because $\bv|_{\p \Omega_h} = 0$, we have $\|\bv\|_{L^2(T)}\le C h_T \|\nab \bv\|_{L^2(T)}$ for $T\in \calT_h^\p$.

Thus, recalling that $\bE_h\bv \in \bH_0^1(\Omega_h)$, we may apply the triangle inequality, Lemma \ref{Eh} (twice with $m=0$), and the Poincar\'e inequality (twice) to determine
\begin{align*}
\|\bv\|_{L^2(\Omega_h)}^2 \leq
& 2\left(\|\bE_h \bv\|_{L^2(\Omega_h)}^2+ \sum_{T \in \mathcal{T}_h^{\partial}} \|\bv - \bE_h \bv\|_{L^2(T)}^2 \right) \\
\leq& C \left( \|\nab \bE_h \bv\|_{L^2(\Omega_h)}^2+ \sum_{T \in \mathcal{T}_h^{\partial}} h_T^2 \|\nabla\bv\|_{L^2(T)}^2 \right) \\
\leq& C \left( \|\nab_h  \bv\|_{L^2(\Omega_h)}^2+ \sum_{T \in \mathcal{T}_h^{\partial}} \|\nabla(\bv-\bE_h \bv) \|_{L^2(T)}^2 \right) \\
% \leq& C\left(h_T^2 \|\nabla_h \bv\|_{L^2(\Omega_h)}^2 + 
% \|\nabla_h(\bE_h \bv - \bv)\|_{L^2(\Omega_h)}^2 
% + \|\nabla_h \bv\|_{L^2(\Omega_h)}^2\right) \\
% \leq& C\|\nabla \bv\|_{L^2(\Omega_h)}^2 + C \sum_{T \in \mathcal{T}_h^{\partial}} h_T^2\|\bv\|_{H^1(T)}^2  \\
 \leq& C\|\nabla_h \bv\|_{L^2(\Omega_h)}^2. 
\end{align*}
\end{proof}

Using the $H^1$-conforming relative in Lemma \ref{Eh}
and the fact that the Lagrange DOFs
are Gauss-Lobatto nodes,
we show that functions in $\bV_k^h$
possess weak continuity properties across interior edges.
To describe the result, we set $\calE_h^I$ to denote
the set of interior edges of $\calT_h$, and define the jump
of a vector-valued function across an edge $e = \p T_+\cap \p T_-\in \calE_h^I\  (T_{\pm}\in \calT_h)$ as
\[
[\bv]|_e  = \bv_+\otimes \bn_+|_e + \bv_-\otimes \bn_-|_e,
\]
where  $\bv_{\pm} = \bv|_{T_{\pm}}$ and $\bn_{\pm}$ is the outward unit normal
of $\p T_{\pm}$ restricted to $e$.
%%%%%%%%%%%%%%%%%%%%%%%%
%%%%%%%%%%%%%%%%%%%%%%%%
%%%%%%%%%%%%%%%%%%%%%%%%
\begin{lemma}\label{jumpBound}
Let $\bw\in \bH^s(\Omega)$ with $s\ge 2$,
and set $r = \min\{s-1,k-1\}$. We extend 
$\bw$ to $\mathbb{R}^2$ such that
$\|\bw\|_{H^{r+1}(\bbR^2)}\le C \|\bw\|_{H^{r+1}(\Omega)}$.
Then there holds for all 
%There holds, for all {\color{red}{$\bw\in \bH^s(\Omega \cup \Omega_h)\ (s\ge 2)$}},  
$\bv \in \bV^h_k$, and $m=0,1,\ldots,k$,
\begin{equation}\label{eqn:jumpBound}
\bigg|\sum_{e\in \mathcal{E}_h^I} \int_e \nab \bw: [\bv] \bigg| \leq C h^{r+m} \|\bw\|_{H^{r+1}(\Omega)} \|\bv\|_{H^m_h(\Omega_h)}. \quad
\end{equation}
\end{lemma}
\begin{proof}
For $e\in \mathcal{E}_h^I$, let $T_+,T_-\in \calT_h$
such that $e = \p T_+\cap \p T_-$.
We let $G_e\in [H^1(T_+\cup T_-)]^{2\times 2}$
such that $G_e|_{T_\pm}\circ F_{T_\pm}\in [\pol_{k-2}(\hat T)]^{2\times 2}$ and 
\[
\int_{T_+\cup T_-} G_e:Q = \int_{T_+\cup T_-} \nab \bw:Q 
\]
for all $Q\in [H^1(T_+\cup T_-)]^{2\times 2}$ 
with $Q|_{T_\pm}\circ F_{T_\pm}\in [\pol_{k-2}(\hat T)]^{2\times 2}$.
That is $G_e$ is the $L^2(T_+\cup T_-)$ projection of $\nab \bw$
with respect to the local $(k-2)$-degree Lagrange (isoparametric) finite element space.
Note that because $F_{T_{\pm}}$ is affine on the interior edge $e$, 
there holds $G_e|_e\in [\pol_{k-2}(e)]^{2\times 2}$.
We also have by standard approximation theory,
\begin{align}\label{eqn:GeApprox}
    \|\nab \bw-G_e\|_{H^m(T_{\pm})}\le C h^{r-m}_{T}\|\nab \bw\|_{H^{r}(T_+\cup T_-)}\le C h^{r-m}_T\|\bw\|_{H^{r+1}(T_+\cup T_-)}\quad m=0,1,\ldots,r,
\end{align}
% For example, $G_e$ can be taken to be the (local) $L^2(T_+\cup T_-)$ projection
% of $G_e$   By \eqref{eqn:GeApprox} 
where $h_T = \max\{h_{T_+},h_{T_-}\}$.
Thus, by a trace inequality, 
\begin{align}\label{eqn:GeApprox2}
 \|\nab \bw-G_e\|_{L^2(e)}\le C h_{T}^{r-1/2} \|\bw\|_{H^{r+1}(T_+\cup T_-)}.
\end{align}
We then write
\begin{equation}\label{eqn:BreakUp}
\begin{split}
\bigg|\sum_{e\in \mathcal{E}_h^I} \int_e \nab \bw: [\bv] \bigg| 
&\le \bigg|\sum_{e\in \mathcal{E}_h^I} \int_e (\nab \bw-G_e): [\bv-\bE_h \bv] \bigg| 
+ \bigg|\sum_{e\in \mathcal{E}_h^I} \int_e G_e: [\bv] \bigg|\\
&=:I_1+I_2.
\end{split}
\end{equation}
To estimate $I_1$, we use \eqref{eqn:GeApprox2},
Lemma \ref{Eh}, and a trace inequality:
\begin{equation}\label{eqn:I1Bound}
\begin{split}
    I_1
    &\le \left(\sum_{e\in \mathcal{E}_h^I} h_e \|\nab \bw-G_e\|_{L^2(e)}^2\right)^{1/2}
    \left(\sum_{e\in \mathcal{E}_h^I} h_e^{-1} \|\bv-\bE_h \bv\|_{L^2(e)}^2\right)^{1/2}\\
    &\le C \left(\sum_{T\in \calT_h} h_T^{2r} \|\bw\|_{H^{r+1}(T)}^2\right)^{1/2}
     \left(\sum_{T\in \calT_h} h_T^{2m} \|\bv\|_{H^m(T)}^2\right)^{1/2}
     \le C h^{r+m}\|\bw\|_{H^{r+1}(\Omega)}\|\bv\|_{H^m_h(\Omega_h)}.
\end{split}
\end{equation}

To estimate $I_2$, we first observe that,
by construction, for $\bv \in \bV^h_k$ and $e\in \calE_h^I$, we have $[\bv]|_{e}(a) = 0$ for all $a\in \calN_k(T)$ with $a\in \bar e$.
Recalling that these edge degrees of freedom are placed
at Gauss-Lobatto nodes,
it follows from the error of the $(k+1)$-point Gauss-Lobatto rule and the fact that $G_e$ is a polynomial of degree 
$(k-2)$ on $e$ that
%, we have for each $e\in \mathcal{E}_h^I$,
%
    \begin{equation}\label{simpson}
    \begin{split}
     \left| \int_e [\bv]:G_e \right| 
     &\leq C |e|^{2k+1} \left| [\bv]:G_e\right|_{W^{2k,\infty}(e)}\\
     &\leq C h_T^{2k+1} \bigg(\|\bv \|_{W^{2k,\infty}(K_+)} + \| \bv \|_{W^{2k,\infty}(K_-)}\bigg)\|G_e\|_{W^{k-2,\infty}(e)}\quad \forall e\in \calE_h^I,
     \end{split}
    \end{equation}
     where $K_{\pm} \in T_{\pm}^{ct}$ share edge $e$.

 Note that a standard inverse/trace estimate yields
 \begin{align}
     \|G_e\|_{W^{k-2,\infty}(e)}\le C h_{T}^{-1}\|G_e\|_{H^{k-2}(T_\pm)}.\label{invTrace}
     \end{align}
From here, we consider two cases:

\textit{Case 1: $k-2 \leq r$}. For this case, we recall that $r = \min\{s-1,k-1\}$, and so $r-k\le -1$. It therefore holds that we have $h^{-1}_T \leq h^{r-k}_T$. With this, \eqref{eqn:GeApprox}, 
and $(k-1)\le (r+1)$, we have
\begin{align*}
\|G_e\|_{W^{k-2,\infty}(e)}\le& C h_{T}^{-1}\big(\|G_e - \nabla \bw\|_{H^{k-2}(T_\pm)} + \|\bw \|_{H^{k-1}(T_\pm)}\big) \nonumber \\
\le& C h_{T}^{r -k} \|\bw \|_{H^{r + 1}(T_+\cup T_-)}.
\end{align*}

\textit{Case 2: $r\leq k-2$}. For the second case, we may apply another inverse estimate to \eqref{invTrace} before applying \eqref{eqn:GeApprox}. This yields
\begin{alignat*}{1}
\|G_e\|_{W^{k-2,\infty}(e)} 
\leq& Ch_{T}^{-1}h_{T_\pm}^{r-(k-2)}\|G_e\|_{H^r(T_\pm)} \\
\leq& Ch_{T}^{r-k+1}\big(\|G_e - \nabla \bw\|_{H^r(T_\pm)} + \|\bw\|_{H^{r+1}(T_\pm)}\big) \\
\leq& Ch_{T}^{r-k+1}\|\bw\|_{H^{r+1}(T_+\cup T_-)}.
\end{alignat*}

Consequently, we take the less sharp estimate in these cases in \eqref{simpson} 
and apply the inverse estimates \eqref{eqn:InvI1}-\eqref{eqn:InvI2} to $\|\bv\|_{W^{2k,\infty}(K_{\pm})}$ 
          to obtain
     \begin{align*}
      \left| \int_e [\bv]:G_e \right|   
      &\leq C h_T^{k+r} \bigg(\|\bv \|_{H^{2k}(K_+)} + \| \bv \|_{H^{2k}(K_-)}\bigg)\|\bw\|_{H^{r+1}(T_+\cup T_-)} \\
      &\le  C h_T^{k+r} \bigg(\| \bv \|_{H^{k}(K_+)} + \|\bv \|_{H^{k}(K_-)}\bigg)\|\bw\|_{H^{r+1}(T_+\cup T_-)} \\
      &\le C h_T^{r+m} \bigg(\| \bv \|_{H^{m}(K_+)} + \|\bv \|_{H^{m}(K_-)}\bigg)\|\bw\|_{H^{r+1}(T_+\cup T_-)}.
     \end{align*}

Summing this expression over $\mathcal{E}_h^I$
we obtain an upper bound for $I_2$:
\begin{align}\label{eqn:I2Bound}
    I_2\le C h^{r+m}\|\bw\|_{H^{r+1}(\Omega)} \|\bv\|_{H^m_h(\Omega_h)}.
\end{align}
Applying the estimates \eqref{eqn:I1Bound} and \eqref{eqn:I2Bound}
towards \eqref{eqn:BreakUp} yields the result.
   \end{proof}

\begin{remark} \label{remEq}
We note that the result above is not as sharp if Newton-Cotes (uniformly spaced) nodes are used instead of Gauss-Lobatto nodes. Indeed, Newton-Cotes integration on $m$ points is exact on polynomials in $\pol^m$, if $m$ is odd, and $\pol^{m-1}$ if $m$ is even, so the bound on the right-hand side of \eqref{simpson} becomes 
\begin{equation*}
    C h_T^{k+2} \bigg(\|\bv \|_{W^{k+1,\infty}(K_+)} + \| \bv \|_{W^{k+1,\infty}(K_-)}\bigg)\|G_e\|_{W^{k-2,\infty}(e)}\quad \forall e\in \calE_h^I,
\end{equation*}
if $k$ is odd, and
\begin{equation*}
C h_T^{k+3} \bigg(\|\bv \|_{W^{k+2,\infty}(K_+)} + \| \bv \|_{W^{k+2,\infty}(K_-)}\bigg)\|G_e\|_{W^{k-2,\infty}(e)}\quad \forall e\in \calE_h^I,
\end{equation*}
if $k$ is even. Thus, if we use equidistant points, the bound loses $k-1$ powers of $h$ if $k$ is odd, and $k-2$ if $k$ is even.
\end{remark}

%%%%%%%%%%%%%%%%%%%%%%%%%%%%%%%%%%%%%%%%%%%%%%%%%%%%%%%%%%%%%%%%
%%%%%%%%%%%%%%%%%%%%%%%%%%%%%%%%%%%%%%%%%%%%%%%%%%%%%%%%%%%%%%%%
%%%%%%%%%%%%%%%%%%%%%%%%%%%%%%%%%%%%%%%%%%%%%%%%%%%%%%%%%%%%%%%%
\subsection{Inf-sup stability}

An inf-sup stability result for the finite element pair $\bV_k^h\times Q_k^h$ was proven in \cite[Theorem 4.4]{NeilanOtus21} in the case $k=2$.
The argument given there are essentially valid for all $k\ge 2$. Consequently, we only provide a sketch
of the proof in the general case.

%%%%%%%%%%%%%%%%%%%%%%%%%%
\begin{theorem}\label{infsup}
    There holds
    \begin{equation}\label{eqn:InfSup}
        \sup_{\bv \in \bV^h_k\setminus \{0\}} \frac{\int_{\Omega_h} (\dive{\bv}) q}{\|\nabla_h \bv\|_{L^2(\Omega_h)}} \geq C\|q\|_{L^2(\Omega_h)} \quad \forall q \in Q^h_{k-1}.
    \end{equation}
\end{theorem}

\subsubsection*{Sketch of proof for Theorem \ref{infsup}} 
Fix $q\in Q_{k-1}^h$, and set $\bar q\in Q_{k-1}^h$ to be piecewise constant with respect to $\calT_h$
satisfying $\int_T (q-\bar q)/\det(DF_T\circ F_T^{-1}) = 0$ for all $T\in \calT_h$.
By a change of variables, we see that $(q-\bar q)\circ F_T\in \hat Q_{k-1,0}$.

Next, the results in, e.g., \cite{guzman2018inf} show that
$\widehat{\rm div}:\hat \bV_{k,0}\to \hat Q_{k-1,0}$ is surjective with bounded
right inverse.
Consequently, for each $T\in \calT_h$, there exists $\hat \bv_{1,T} \in \hat \bV_{k,0}$
such that $\widehat{\rm div}\,\hat \bv_{1,T} = h_T^2 (q-\bar q)|_T \circ F_T$.
and $\|\hat \bv_{1,T}\|_{H^1(\hat T)}\le C h_T^2 \|(q-\bar q)|_T\circ F_T\|_{L^2(\hat T)}\le C h_T \|q-\bar q\|_{L^2(T)}$.
Setting $\bv_{1,T} = (A_T\hat \bv)\circ F_T^{-1}\in \bV_{k,0}$,
we have ${\rm div}\,\bv_{1,T} = h^2_T(q-\bar q)/(\det(DF_T\circ F_T^{-1}))$ by the divergence-preserving
properties of the Piola transform, and $\|\nab \bv_{1,T}\|_{L^2(T)}\le C \|q-\bar q\|_{L^2(T)}$ by a scaling argument.

We then define $\bv_1\in \bV_k^h$ such that $\bv_1|_T = \bv_{1,T}$ for all $T\in \calT_h$. Thus $\|\bv_1\|_{L^2(\Omega_h)}\le C\|q-\bar q\|_{L^2(\Omega_h)}$,
and
\begin{align*}
\sup_{\bv\in\bV_k^h \backslash \{0\}} \frac{\int_{\Omega_h}(\dive \bv)q}{\|\nabla \bv\|_{L^2(\Omega_h)}}
&\ge \frac{\int_{\Omega_h}(\dive \bv_1)q}{\|\nabla \bv_1\|_{L^2(\Omega_h)}}
= \frac{\int_{\Omega_h}(\dive \bv_1)(q-\bar q)}{\|\nabla \bv_1\|_{L^2(\Omega_h)}}\\
& = \frac{\sum_{T\in \calT_h} h_T^2 \int_T |q-\bar q|^2/(\det(DF_T\circ F_T^{-1}))}{\|\nab \bv_1\|_{L^2(\Omega_h)}}\\
&\ge \gamma_0 \|q-\bar q\|_{L^2(\Omega_h)}.
\end{align*}

%For each $T\in \calT_h$, there exists $\bv_{2,T}\in \bV_{k,0}(T)$ satisfying
%$({\rm div}\,\bv_{2,T}) = h_T (q-\bar q) /\det(DF_T\circ F_T^{-1})$ on $T$ and $\|\bv_{2,T}\|_{H^1(T)}\le C \|q\|_{L^2(T)}$.

Next, Theorem 4.4 in \cite{NeilanOtus21} shows
that
\[
\sup_{\bv\in \bV_k^h\backslash \{0\}} \frac{\int_{\Omega_h} ({\rm div}\, \bv) \bar q}{\|\nab_h \bv\|_{L^2(\Omega_h)}} \ge \gamma_1 \|\bar q\|_{L^2(\Omega_h)}.
\]
Consequently, it follows that
\begin{alignat*}{1}
\|q\|_{L^2(\Omega_h)} \leq& \|q - \bar{q}\|_{L^2(\Omega_h)} + \|\bar{q}\|_{L^2(\Omega_h)} \\
\leq& (\gamma_0^{-1}+\gamma_1^{-1}(1+\gamma_0^{-1})) \sup_{\bv\in\bV_h^k \setminus \{0\}} \frac{\int_{\Omega_h}(\dive \bv)q}{\|\nabla \bv\|_{L^2(\Omega_h)}}.
\end{alignat*}
 
%%%%%%%%%%%%%%%%%%%%%%%%%%%%%%%%%%%%%%%%%%%
%%%%%%%%%%%%%%%%%%%%%%%%%%%%%%%%%%%%%%%%%%%
%%%%%%%%%%%%%%%%%%%%%%%%%%%%%%%%%%%%%%%%%%%
\section{The Stokes System and Finite Element Method}\label{stokesFEM}

We let $(\bu,p) \in \bH_0^1(\Omega)\times L_0^2(\Omega)$ be the solution to the Stokes problem
\begin{equation}\label{fullStokes}
    \begin{cases}
        -\nu\Delta \bu + \nabla p = \bff \quad &\text{ in } \Omega, \\
        \dive{\bu} = 0  \quad &\text{ in } \Omega, \\
        \bu = 0 \quad &\text{ on } \partial \Omega,
    \end{cases}
\end{equation}
where $\nu>0$ is the viscosity. We assume
the domain $\Omega$ and source function ${\bm f}$ are sufficiently
smooth such that $(\bu,p)\in \bH^s(\Omega)\times H^{s-1}(\Omega)$
with $s\ge 2$.  We then extend the velocity solution
to $\mathbb{R}^2$ such that the extension (still denoted by $\bu$) 
is divergence-free and satisfies \cite{kato2000}
\begin{equation}\label{spaceExtend}
\|\bu\|_{H^s(\mathbb{R}^2)} \leq  C\|\bu\|_{H^s(\Omega)}.
\end{equation}
Likewise, we extend the pressure solution $p$ to $\mathbb{R}^2$ with $\|p\|_{H^{s-1}(\bbR^2)}\le C \|p\|_{H^{s-1}(\Omega)}$
and extend the source function by setting ${\bm f}= -\nu \Delta \bu+\nab p$ in $\bbR^2$.

We define the continuous bilinear forms
\begin{alignat*}{1}
a(\bu,\bv):=& \int_{\Omega} \nu \nabla \bu : \nabla \bv, \qquad
b(\bv,p) := -\int_{\Omega}(\dive{\bv})p,
\end{alignat*}
and the discrete bilinear forms
\begin{alignat*}{1}
a_h(\bu_h,\bv):=& \int_{\Omega_h} \nu \nabla_h \bu_h : \nabla_h \bv,\qquad
b_h(\bv,p_h) := -\int_{\Omega_h}(\dive{\bv})p_h.
\end{alignat*}

Clearly, the solution to \eqref{fullStokes} solves the variational problem
\begin{equation}\label{continuousVar}
    a(\bu,\bv) = \int_{\Omega}\bff \cdot \bv \quad \forall \bv \in \bX:=\{\bv\in \bH^1_0(\Omega):\ \dive{\bv}=0\}.
\end{equation}

We  define the finite element method as finding $(\bu_h, p_h) \in \bV^h_k \times Q^h_{k-1}$ such that 
\begin{subequations}
    \label{eqn:FEM}
\begin{alignat}{2}
a_h(\bu_h,\bv) + b_h(\bv,p_h) =& \int_{\Omega_h} \bff_h \cdot \bv \quad &&\forall \bv \in \bV^h_k,\label{FE11}\\
-b_h(\bu_h,q) =& 0 \quad && \forall q \in Q^h_k,\label{FE12}
\end{alignat}
\end{subequations}
where $\bff_h\in \bL^2(\Omega_h)$ is a suitable (and computable) approximation to $\bff|_{\Omega}$. 
It follows from the inf-sup condition in Theorem \ref{infsup}
and the Poincare inequality in Corollary \ref{poincare}
that there exists a unique solution to \eqref{eqn:FEM}.
Moreover, by a simple generalization of \cite[Lemma 5.2]{NeilanOtus21},
the method \eqref{eqn:FEM} yields divergence-free velocity approximations.
\begin{lemma}\label{lem:DivFree}
    Let $\bu_h\in \bV^h_k$ satisfy \eqref{FE12}.
    Then $\dive \bu_h=0$ in $\Omega_h$.
\end{lemma}

%%%%%%%%%%%%%%%%%%%%%%%%%%%%%%%%%%%%%%%%%%%%%%%%%%%%%%%%%
\subsection{Energy estimates}
In this section, we derive error estimates
for the approximation velocity and pressure solutions
in the $H^1$ and $L^2$ norms, respectively.
To this end, we  define the discrete space
of divergence-free functions
\begin{equation*}
\bX^h_k := \{ \bv \in \bV^h_k : \dive{\bv} = 0\} \nsubseteq \bX := \{\bv \in \bH_0^1(\Omega): \dive{\bv} = 0\},
\end{equation*}
and note that functions in this space are not necessarily in $\bH_0^1$.
Lemma \ref{lem:DivFree} shows that $\bu_h\in \bX_h^k$,
and thus the velocity solution $\bu_h$ is uniquely characterized
as the solution of the Poisson-type problem
\begin{equation}\label{discreteVar}
a_h(\bu_h,\bv) = \int_{\Omega_h}\bff_h \cdot \bv \quad \forall \bv \in \bV^h_k.
\end{equation}
%
%Additionally, we have the following error bounds following the results and methods of \cite{NeilanOtus21}
%%%%%%%%%%%%%%%%%%%%%%%%%%%%%%   
%%%%%%%%%%%%%%%%%%%%%%%%%%%%%%   
%%%%%%%%%%%%%%%%%%%%%%%%%%%%%%   
\begin{theorem}\label{H1}
Let $(\bu,p)\in \bH^s(\Omega)\times H^{s-1}(\Omega)$ satisfy \eqref{fullStokes} with $s \ge 2$.
Then there holds
\begin{equation}\label{eqn:velocityH1}
\|\nabla_h(\bu - \bu_h)\|_{L^2(\Omega_h)} \leq C\left( h^{\ell-1} \|\bu\|_{H^{\ell}(\Omega)} + \nu^{-1} \vert \bff - \bff_h \vert_{\bX_k^*} \right),
\end{equation}
where  $\ell = \min\{k+1,s\}$, and 
\begin{equation*}
    \vert \bff - \bff_h \vert_{\bX_k^*} = \sup_{\bv \in \bX_k^h \setminus \{0\}} \frac{\int_{\Omega_h}(\bff-\bff_h) \cdot \bv}{\|\nabla_h \bv \|_{L^2(\Omega_h)}}.
\end{equation*}
The pressure approximation, $p_h$, satisfies
\begin{alignat}{1}\label{eqn:pressureL2}
    \|p-p_h\|_{L^2(\Omega_h)} 
    %\leq& C\bigg(\nu \|\nabla(\bu - \bu_h)\|_{L^2(\Omega_h)} + \nu h^{\ell-1}\|\bu\|_{H^{\ell}(\Omega_h)} \\
    %&+ \inf_{q \in Q^h} \|p - q \|_{L^2(\Omega_h)} + \|\bff- \bff_h\|_{L^2(\Omega_h)}\bigg),
    \leq& C\left( h^{\ell-1}(\nu \|\bu\|_{H^{\ell}(\Omega)} +\|p\|_{H^{\ell-1}(\Omega)})
     + \|\bff- \bff_h\|_{L^2(\Omega_h)}\right).
\end{alignat}
%again for $m=1,2$.
\end{theorem}
\begin{proof}
%The case where $k=2$ and $s=3$ is proved in \cite{NeilanOtus21}, and we closely follow the work 
%in that paper to prove the general case.

From  standard theory of non-conforming finite elements (see, for example, \cite{brenner2008mathematical})
and the inf-sup condition \eqref{eqn:InfSup}, 
%we have Strang's Second Lemma applied to our system.
\begin{equation}
\label{strang2}
\begin{split}
    \nu \|\nabla_h(\bu - \bu_h)\|_{L^2(\Omega_h)} \leq&  \inf_{\bw \in \bX^h_k} \nu \|\nabla_h(\bu - \bw)\|_{L^2(\Omega_h)} + \sup_{\bv \in \bX^h_k \setminus \{0\}} \frac{a_h(\bu_h - \bu, \bv)}{\|\nabla_h \bv\|_{L^2(\Omega_h)}} \\
    \leq& C\inf_{\bw \in \bV^h_k} \nu \|\nabla_h(\bu - \bw)\|_{L^2(\Omega_h)} + \sup_{\bv \in \bX^h_k \setminus \{0\}} \frac{a_h(\bu_h - \bu, \bv)}{\|\nabla_h \bv\|_{L^2(\Omega_h)}} \\
%    \leq&  C\nu \|\nabla(\bu - J_T \bu)\|_{L^2(\Omega_h)} + \sup_{\bv \in \bX^h_k \setminus \{0\}} \frac{a_h(\bu_h - \bu, \bv)}{\|\nabla \bv\|_{L^2(\Omega_h)}} \nonumber\\
    \leq& C h^{\ell-1} \nu \|\bu \|_{H^\ell(\Omega)}+ \sup_{\bv \in \bX^h_k \setminus \{0\}} \frac{a_h(\bu_h - \bu, \bv)}{\|\nabla_h\bv\|_{L^2(\Omega_h)}}, 
\end{split}
\end{equation}
where the final step follows from Lemma \ref{interp}.

To address the consistency term, we note that we have $\forall \bv \in \bX^h_k$
\begin{equation}
\label{strang2A}
\begin{split}
a_h(\bu_h - \bu,\bv) %=& \int_{\Omega_h}\nabla(\bu_h - \bu):\nabla \bv \\
=& \int_{\Omega_h} \bff \cdot \bv - a_h(\bu, \bv) + \int_{\Omega_h}(\bff_h - \bff)\cdot \bv \\
%=& - \nu \int_{\Omega_h} \Delta \bu \cdot \bv + \int_{\Omega_h}\nabla p \cdot \bv - a_h(\bu,\bv) + \int_{\Omega_h}(\bff_h - \bff)\cdot \bv \nonumber\\
%=& -\nu\int_{\Omega_h} \Delta \bu \cdot \bv - \int_{\Omega_h} p (\dive{\bv}) + \int_{\partial \Omega_h} p \bn \cdot \bv - a_h(\bu,\bv) + \int_{\Omega_h}(\bff_h - \bff)\cdot \bv \nonumber\\
=&  -\nu\int_{\Omega_h} \Delta \bu \cdot \bv  - a_h(\bu,\bv) + \int_{\Omega_h}(\bff_h - \bff)\cdot \bv . 
\end{split}
\end{equation}
Note that the last step uses the fact that $\bv \in \bX^h_k$, therefore $\dive{\bv} = 0$ and $\bv = 0$ on $\partial \Omega_h$.

% From integration by parts, we have
% \begin{alignat*}{1}
% a_h(\bu,\bv) =& -\nu \int_{\Omega_h} \Delta \bu \cdot \bv + \nu \sum_{e\in \mathcal{E}^{I,\partial}} \int_e \nabla \bu : [\bv]
% \end{alignat*}

%Substituting this in to \eqref{strang2A}, 
We then apply a standard integration-by-parts formula in \eqref{strang2A},
Lemma \ref{jumpBound} (with $m=1$, and noting $r = \min\{s-1,k-1\}\le \ell-1$) and Corollary \ref{poincare}
to obtain
%we have
\begin{equation} 
\label{strang2B}
\begin{split}
a_h(\bu_h - \bu,\bv) 
=& -\nu \sum_{e\in \mathcal{E}^{I}_h} \int_e \nabla \bu : [\bv] + \int_{\Omega_h}(\bff_h - \bff)\cdot \bv\\
\le& C \nu  h^{\ell-1} \|\bu\|_{H^\ell(\Omega)} \|\nab_h \bv\|_{L^2(\Omega_h)}+ \|\bff_h - \bff\|_{X_k^*} \|\nab_h \bv\|_{L^2(\Omega_h)}.
\end{split}
\end{equation}
Finally, to complete the velocity bound \eqref{eqn:velocityH1}, we apply this estimate to \eqref{strang2}.

To prove the pressure bound \eqref{eqn:pressureL2}, we fix $q\in Q^h_{k-1}$. 
For any $\bv \in \bV^h_k$, we then have the following identity, using integration by parts and \eqref{fullStokes}:
\begin{equation}
\label{pressure1}
\begin{split}
\int_{\Omega_h} (\dive{\bv})(p_h - q) 
%=& a_h(\bu_h,\bv) - \int_{\Omega_h}(\dive{\bv})q - \int_{\Omega_h}\bff_h\cdot \bv \\
=& a_h(\bu_h,\bv) - \int_{\Omega_h}(\dive{\bv})q - \int_{\Omega_h}(\bff_h - \bff)\cdot \bv -\int_{\Omega_h}\bff \cdot \bv \\
=& a_h(\bu_h,\bv) + \nu\int_{\Omega_h} \Delta \bu\cdot \bv - \int_{\Omega_h} \nabla p \cdot \bv - \int_{\Omega_h}(\dive{\bv})q  - \int_{\Omega_h}(\bff_h - \bff)\cdot \bv \\
=& a_h(\bu_h - \bu, \bv) - \int_{\Omega_h}(\dive{\bv})(q-p) - \int_{\Omega_h} (\bff_h - \bff)\cdot \bv + \nu \sum_{e\in \mathcal{E}^{I}_h} \int_e \nabla \bu : [\bv]. 
\end{split}
\end{equation}
Then, applying \eqref{eqn:jumpBound} to \eqref{pressure1} and Corollary \ref{poincare}, we have
% \begin{alignat}{1}
% \int_{\Omega_h} (\dive{\bv})(p_h - q) \leq& \nu\|\nabla(\bu_h - \bu)\|_{L^2(\Omega_h)}\|\nabla\bv\|_{L^2(\Omega_h)} + \|q-p\|_{L^2(\Omega_h)}\|\nabla \bv\|_{L^2(\Omega_h)} \nonumber \\
% & + \nu h \|\bu\|_{H^2(\Omega_h)}\|\nabla \bv\|_{L^2(\Omega_h)} + \|\bff - \bff_h\|_{L^2(\Omega_h)} \|\bv\|_{L^2(\Omega_h)} \label{pressure2}
% \end{alignat}
%
% Applying Corollary \ref{poincare}, the above inequality \eqref{pressure2} becomes
\begin{equation}
\label{pressure3} 
\begin{split}
\int_{\Omega_h} (\dive{\bv})(p_h - q)\leq& C\bigg(\nu\|\nabla_h(\bu_h - \bu)\|_{L^2(\Omega_h)} + \|q-p\|_{L^2(\Omega_h)}  \\
& \quad + \nu h^{\ell-1} \|\bu\|_{H^\ell(\Omega_h)} + \|\bff - \bff_h\|_{L^2(\Omega_h)} \bigg)\|\nabla_h \bv\|_{L^2(\Omega_h)}. 
\end{split}
\end{equation}

Finally, by triangle inequality and Theorem \ref{infsup}, we have
\begin{alignat*}{1}
\|p - p_h\|_{L^2(\Omega_h)} \leq& \|p - q\|_{L^2(\Omega_h)} + \|p_h - q\|_{L^2(\Omega_h)} \\
\leq& \|p - q\|_{L^2(\Omega_h)} + \sup_{\bv \in \bV^h \setminus \{0\}} \frac{\int_{\Omega_h} (\dive{\bv})(p_h - q)}{\|\nabla \bv\|_{L^2(\Omega_h)}}.
\end{alignat*}
Applying \eqref{pressure3} to this result, taking the infimum over $q \in Q^h_{k-1}$, and using \eqref{eqn:velocityH1} completes the proof.
\end{proof}

%%%%%%%%%%%%%%%%%%%%%%%%%
%%%%%%%%%%%%%%%%%%%%%%%%%
%%%%%%%%%%%%%%%%%%%%%%%%%
\begin{corollary} \label{H2Bound}
Assume the conditions in Theorem \ref{H1} are satisfied,
and in addition, assume the mesh $\calT_h$ is quasi-uniform.  
Then the solution $\bu_h\in \bV^h_k$ to \eqref{eqn:FEM} satisfies
% Then if $\bu \in \bH^2(\Omega)$ is the solution to the 
%     Stokes equation \eqref{fullStokes} with forcing term $\bff \in L^2(\Omega)$, 
%     and if $\bu_h \in \bV^h_k$ is its approximation solving \eqref{FE12}, we have
    \begin{equation*}
        \|\bu_h\|_{H^\ell_h(\Omega_h)} \leq C \bigg(\| \bu\|_{H^\ell(\Omega)} + h^{1-\ell}\nu^{-1}\vert \bff - \bff_h\vert_{\bX_k^*}\bigg).
    \end{equation*}
\end{corollary}
\begin{proof}
Define $\bI_k^h \bu\in \bV_k^h$ to be the approximation to $\bu$ 
given in Lemma \ref{interp}.
%\textcolor{red}{such that $\bI_h$ satisfies} $\bI_h \bu|_T = \bI_T \bu$ for all $T\in \calT_h$.
%To begin, we let $\Pi_h$ be a projection operator such that $\Pi_h \bu \in \bV^h_k$. Then we have 
%\begin{alignat*}{1}
%    \|\bu_h\|_{H^2(\Omega_h)} \leq& \|\bu_h - \bu \|_{H^2(\Omega_h)} + \|\bu\|_{H^2(\Omega)} \\
%    \leq& \|\bu - \Pi_h \bu \|_{H^2(\Omega_h)} + \|\Pi_h \bu - \bu_h\|_{H^2(\Omega_h)} + \|\bu\|_{H^2(\Omega)}.
%\end{alignat*}
Then, applying the inverse inequality \eqref{eqn:InvI1},  Lemma \ref{interp}, and Theorem \ref{H1}, we have
\begin{alignat*}{1}
 \|\bu_h\|_{H^\ell_h(\Omega_h)} 
 \leq& C\big(\|\bu - \bI_k^h \bu \|_{H^\ell_h(\Omega_h)} + h^{1-\ell}\|\bI_k^h \bu - \bu_h\|_{H^1_h(\Omega_h)} + \|\bu\|_{H^\ell(\Omega)}\big) \\
 \leq& C\big(\|\bu - \bI_k^h \bu \|_{H^\ell_h(\Omega_h)} + h^{1-\ell}\|\bI_k^h \bu - \bu\|_{H^1_h(\Omega_h)} + h^{1-\ell}\|\bu - \bu_h\|_{H^1_h(\Omega_h)} +\|\bu\|_{H^\ell(\Omega)}\big)\\
 \leq& C\big(\|\bu\|_{H^\ell(\Omega)} + h^{1-\ell}\nu^{-1}\vert \bff - \bff_h\vert_{\bX_k^*}\big).
\end{alignat*}
\end{proof}
\begin{remark}
If $\bff$ is sufficiently smooth, and $\bff_h$ is, for example, the $k$th degree nodal (isoparametric) interpolant, then 
$|\bff-\bff_h|_{\bX_k^*}\le \|\bff-\bff_h\|_{L^2(\Omega_h)}\le C h^{k+1}\|\bff\|_{H^{k+1}(\Omega_h)}$.
Thus, Corollary \ref{H2Bound} yields
%we may extend this result to 
\begin{equation}\label{uH2_withF}
     \|\bu_h\|_{H^\ell_h(\Omega_h)} \leq C\left( \| \bu\|_{H^\ell(\Omega)} + h^{k-\ell+2}\nu^{-1}\| \bff\|_{H^{k+1}(\Omega_h)}\right).
\end{equation}
\end{remark}
%%%%%%%%%%%%%%%%%%%%%%%%%%%%%%%%%%%%%%%%%%%%%%%%%%%%
\section{Convergence analysis in $L^2$}\label{l2est}
In this section, we prove the following optimal-order 
$L^2$ error estimate.
\begin{theorem}
Assume the conditions in Theorem \ref{H1} are satisfied,
and in addition, assume the mesh $\calT_h$ is quasi-uniform.
    We have 
    \begin{equation}\label{eqn:L2velocity}
        \|\bu-\bu_h\|_{L^2(\Omega_h)} \leq C\bigg( h^\ell \|\bu\|_{H^\ell(\Omega)} + (\nu^{-1} h +1)\vert \bff - \bff_h \vert_{\bX^*_h}\bigg),
    \end{equation}
    where $C$ is a constant that does not depend on the mesh parameter $h$,
    and we recall $\ell = \min\{s,k+1\}$.
\end{theorem}
\begin{proof}
To derive \eqref{eqn:L2velocity}, we first write
\begin{equation}\label{splitUp}
\|\bu - \bu_h\|_{L^2(\Omega_h)} \leq \|\bu - \bu_h\|_{L^2(\Omega_h \setminus \Omega)} + \|\bu - \bu_h\|_{L^2(\Omega_h \cap \Omega)}=:J_1+J_2.
\end{equation}

To bound $J_1$, we introduce $\bE_h: \bV^h_k \to \bH_0^1(\Omega_h)$ as defined in Lemma \ref{Eh}. Consequently, we may write
\begin{equation*}
    J_1 \leq \|\bu - \bE_h \bu_h\|_{L^2(\Omega_h \setminus \Omega)} + \|\bE_h \bu_h - \bu_h \|_{L^2(\Omega_h \setminus \Omega)}.
\end{equation*}
A bound of the second term in this sum follows  from Lemma \ref{Eh} and Corollary \ref{H2Bound}: 
\begin{alignat*}{1}
J_1 \leq& \|\bu - \bE_h \bu_h\|_{L^2(\Omega_h \setminus \Omega)} + \|\bE_h \bu_h - \bu_h \|_{L^2(\Omega_h)} \\
\leq&  \|\bu - \bE_h \bu_h\|_{L^2(\Omega_h \setminus \Omega)} + C h^{\ell} \|\bu_h\|_{H^{\ell-1}_h(\Omega_h)}\\
\leq& \|\bu - \bE_h \bu_h\|_{L^2(\Omega_h \setminus \Omega)} + C \bigg(h^\ell\|\bu\|_{H^\ell(\Omega_h)} + h \nu^{-1} \vert \bff - \bff_h\vert_{\bX_h^*}\bigg) .
\end{alignat*}

To bound the remaining term, begin with H\"older's inequality and recall that $H^1$ embeds in $L^6$ and $k\geq 2$. Thus we have
\begin{alignat*}{1}
\|\bu - \bE_h\bu_h\|_{L^2(\Omega_h \setminus \Omega)} \leq&  |\Omega_h \setminus \Omega|^{1/3} \|\bu-\bE_h \bu_h\|_{L^6(\Omega_h)}\\
\leq& C h^{(k+1)/3} \|\bu - \bE_h\bu_h\|_{L^6(\Omega_h)}\\
\leq& C h \|\bu - \bE_h\bu_h\|_{H^1(\Omega_h)}.
\end{alignat*}
It follows from Theorem \ref{H1}, Lemma \ref{Eh}, and Corollary \ref{H2Bound} that 
\begin{alignat*}{1}
\|\bu - \bE_h\bu_h\|_{L^2(\Omega_h \setminus \Omega)} \leq
& C h\bigg( \|\bu - \bu_h \|_{H^1_h(\Omega_h)} + \|\bu_h - \bE_h\bu_h \|_{H^1_h(\Omega_h)}\bigg) \\
\leq
& C  \bigg(h^{\ell} \|\bu \|_{H^\ell(\Omega)} + {h} \nu^{-1} \vert \bff - \bff_h\vert_{\bX_k^*}\bigg) .
\end{alignat*}
Combining this with the result above, we have
\begin{equation}\label{firstBound}
J_1\leq C\bigg( h^\ell \|\bu\|_{H^\ell(\Omega)}+ h \nu^{-1} \vert \bff - \bff_h\vert_{\bX_k^*}\bigg).
\end{equation}

To bound $J_2$, we let $\bphi\in \bL^2(\Omega\cup \Omega_h)$
such that $\bphi|_{\Omega\cap \Omega_h} = (\bu-\bu_h)|_{\Omega\cap \Omega_h}$
and $\bphi|_{\Omega\cup \Omega_h \backslash (\Omega\cap \Omega_h)}=0$.
We then define $(\bpsi, r) \in \bH^1_0(\Omega) \times L^2_0(\Omega)$ to be the solution to the auxiliary problem
\begin{equation}\label{fullPsi}
\begin{cases}
    -\nu \Delta \bpsi + \nabla r = \bphi \quad&\text{ in } \Omega, \\
    \dive{\bpsi} = 0 \quad&\text{ in } \Omega.
%    \bpsi = 0 \quad &\text{ on }\partial \Omega.
\end{cases}
\end{equation}
Because $\p \Omega$ is smooth and $\bphi|_{\Omega}\in \bL^2(\Omega)$,
there holds $\bpsi\in \bH^2(\Omega)$ with $\|\bpsi\|_{H^2(\Omega)}\le C \|\bphi\|_{L^2(\Omega)} = C \|\bphi\|_{L^2(\Omega\cap \Omega_h)}$
by elliptic regularity theory.
Similar to the solution $\bu$ in \eqref{spaceExtend}, we extend $\bpsi$ to $\mathbb{R}^2$ in such
a way that preserves the divergence--free condition
(cf. \cite{kato2000}) and 
\begin{equation}\label{elliptic}
\|\bpsi\|_{H^2(\mathbb{R}^2)} \leq C\|\bpsi\|_{H^2(\Omega)}\le C \|\bphi\|_{L^2(\Omega\cap \Omega_h)}.
\end{equation}
% Therefore, from elliptic regularity, we have
% \begin{equation}\label{elliptic}
% \|\bpsi\|_{H^2(\mathbb{R}^2)}\leq C\|\bpsi\|_{H^2(\Omega)} \leq C\|\bphi\|_{L^2(\Omega \cap \Omega_h)}.
% \end{equation}

Finally, we define $\bpsi_h \in \bX^h_k$ to be the approximation of $\bpsi$ on $\Omega_h$ satisfying
\begin{equation}\label{psiH}
a_h(\bpsi_h, \bv) = \int_{\Omega_h} \bphi \cdot \bv \quad \forall \bv \in \bV^h_k.
\end{equation}
We note that $\bpsi$ and $\bpsi_h$ are analogous to $\bu$ and $\bu_h$, respectively, in Theorem \ref{H1} when $s=2$ (so that $\ell = 2$), with $\bphi$ replacing both $\bff$ and $\bff_h$. Therefore, the following estimate holds:
\begin{equation}\label{psiH1}
\|\bpsi-\bpsi_h\|_{H^1_h(\Omega_h)}\leq C h\|\bpsi\|_{H^2(\Omega)}\leq C\|\bphi\|_{L^2(\Omega\cap \Omega_h)}.
\end{equation}
Additionally, applying Corollary \ref{H2Bound} yields
\begin{alignat}{1}
\|\bpsi_h\|_{H^2_h(\Omega_h)} \leq& C\|\bpsi\|_{H^2(\Omega)} \leq C\|\bphi\|_{L^2(\Omega\cap \Omega_h)}.
\label{psiHbound}
 \end{alignat}

% To bound the second term in \eqref{splitUp}, we let $\bphi = \bu - \bu_h$ on $\Omega \cap \Omega_h$ (and extend by zero elsewhere, as stated above). Then we have
Next, we write
\begin{equation}
\label{breakUp}
\begin{split}
    (J_2)^2 = \|\bu - \bu_h\|_{L^2(\Omega \cap \Omega_h)}^2 =& \int_{\Omega \cap \Omega_h} \bphi \cdot (\bu - \bu_h) \\
    =& \int_{\Omega} \bphi \cdot \bu - \int_{\Omega_h} \bphi \cdot \bu_h \\
    =& a(\bu, \bpsi) - a_h(\bu_h, \bpsi_h) \\
    =& [ a(\bu,\bpsi) - a_h(\bu, \bpsi)] + a_h(\bu,\bpsi-\bpsi_h) + a_h(\bu-\bu_h, \bpsi_h).
\end{split}
\end{equation}
We now consider the terms of \eqref{breakUp} separately.

\subsubsection*{Bound of $[ a(\bu,\bpsi) - a_h(\bu, \bpsi)]$}
To bound the first terms of \eqref{breakUp}, we begin with
\begin{alignat*}{1}
\vert a(\bu,\bpsi) - a_h(\bu, \bpsi)\vert =& \nu \left| \int_{\Omega} \nabla \bu : \nabla \bpsi - \int_{\Omega_h} \nabla\bu : \nabla \bpsi \right| \\
=& \nu \left| \int_{\Omega \setminus \Omega_h} \nabla \bu : \nabla \bpsi - \int_{\Omega_h \setminus \Omega} \nabla \bu : \nabla \bpsi \right|\\
\leq& C \nu\|\nabla \bu\|_{L^2(\Omega \Delta\Omega_h)}\|\nabla \bpsi\|_{L^2(\Omega \Delta\Omega_h)}.
\end{alignat*}
\begin{comment}
\begin{alignat*}
\leq& \nu \big(\|\nabla \bu\|_{L^\infty(\Omega \setminus \Omega_h)} \|\nabla \bpsi\|_{L^1(\Omega \setminus \Omega_h)} + \|\nabla \bu\|_{L^\infty(\Omega_h \setminus \Omega)} \|\nabla \bpsi\|_{L^1(\Omega_h \setminus \Omega)}\big)\\
\leq& C \nu \|\nabla \bu \|_{L^\infty(\Omega)}\|\nabla \bpsi\|_{L^1((\Omega_h \setminus \Omega) \cup (\Omega \setminus \Omega_h))}
\end{alignat*}
where we again note that integrals over the discrete domain $\Omega_h$ are understood to be the sum of integrals over the elements.

Next, we apply H\"older's inequality with $p = |\log{h}|$ and $q = \frac{p}{p-1}$.
\begin{alignat*}{1}
|a(\bu,\bpsi) - a_h(\bu, \bpsi)| \leq& C \nu \|\nabla \bu\|_{L^\infty(\Omega)}\|\nabla \bpsi\|_{L^p((\Omega_h \setminus \Omega) \cup (\Omega \setminus \Omega_h))} |(\Omega_h \setminus \Omega) \cup (\Omega \setminus \Omega_h)|^{1/q}.
\end{alignat*}

Applying Lemmas \ref{domainSize} and \ref{Cp} and recognizing 
\begin{equation*}
h^{3/q} = e^{-3} h^3,
\end{equation*}
we have \textcolor{red}{(do we get $H^3$ norm of $\bu$ from sobolev embedding of $L^\infty$ below?)}
\begin{alignat*}{1}
|a(\bu,\bpsi) - a_h(\bu, \bpsi)| \leq& C \nu h^3 |\log{h}| \|\nabla \bu\|_{L^\infty(\Omega)} \|\nabla \bpsi\|_{H^1((\Omega_h \setminus \Omega) \cup (\Omega \setminus \Omega_h))} \\
\leq& C \nu h^3 |\log{h}| \|\nabla \bu\|_{L^\infty(\Omega)} \|\bphi\|_{L^2(\Omega \cap \Omega_h)}.
\end{alignat*}
\end{comment}
%%%%%%%%%%%%%%%%%%%%%%%%
%%%%%%%%%%%%%%%%%%%%%%%%
The result of Lemma \ref{domainEdges} implies
\begin{align*}
    \|\nab \bu\|_{L^2(\Omega\Delta\Omega_h)}\leq Ch^{(k+1)/2}\|\bu\|_{H^2(\Omega)},\quad
     \|\nab \bpsi\|_{L^2(\Omega\Delta\Omega)}\leq Ch^{(k+1)/2}\|\bpsi\|_{H^2(\Omega)},
\end{align*}
from which we get
\begin{align}\label{amAh}
    |a(\bu,\bpsi) - a_h(\bu, \bpsi)|\leq Ch^{k+1} \nu\|\bu\|_{H^2(\Omega)}\|\bpsi\|_{H^2(\Omega)}.
\end{align}

\subsubsection*{Bound of $a_h(\bu,\bpsi-\bpsi_h) + a_h(\bu-\bu_h, \bpsi_h)$} 
It now remains to bound the last two terms in \eqref{breakUp}. To begin, we write
\begin{equation}
\begin{split}
a_h(\bu,\bpsi-\bpsi_h) + a_h(\bu-\bu_h, \bpsi_h) =& a_h(\bu - \bu_h, \bpsi - \bpsi_h) + a_h(\bu_h, \bpsi - \bpsi_h) + a_h(\bu-\bu_h, \bpsi_h) \\
\leq& \nu \|\bu-\bu_h\|_{H^1(\Omega_h)}\|\bpsi-\bpsi_h\|_{H^1(\Omega_h)} + a_h(\bu_h, \bpsi - \bpsi_h) + a_h(\bu-\bu_h, \bpsi_h) \\
\leq& C\bigg(\nu h^\ell\|\bu\|_{H^\ell(\Omega)} + h\vert \bff - \bff_h\vert_{\bX^*_h}\bigg)\|\bphi\|_{L^2(\Omega\cap \Omega_h)} \\
&+ a_h(\bu_h, \bpsi - \bpsi_h) + a_h(\bu-\bu_h, \bpsi_h) 
\end{split} \label{start1}
\end{equation}
by Theorem \ref{H1} and \eqref{psiH1}.

Recalling \eqref{strang2B}, we have
by Lemma \ref{jumpBound} (with $m=2$, and noting $\ell = \min\{s,k+1\} \le \min\{s-1,k-1\}+2=r+2$)
\begin{equation}
 \label{ibpU}
\begin{split}
 a_h(\bu-\bu_h, \bpsi_h) 
 %=& \int_{\Omega_h} \nabla (\bu - \bu_h):\nabla \bpsi_h \nonumber\\
% =& -\int_{\Omega_h} \bff \cdot \bpsi_h + a_h(\bu,\bpsi_h) + \int_{\Omega_h} (\bff - \bff_h)\cdot \bpsi_h \nonumber\\
% =& \nu \int_{\Omega_h} \Delta \bu \cdot \bpsi_h - \nu \int_{\Omega_h}\Delta \bu \cdot \bpsi_h -  \nu \sum_{e \in \mathcal{E}^{I,\partial}} \int_e \nabla \bu : [\bpsi_h] + \int_{\Omega_h} (\bff - \bff_h)\cdot \bpsi_h \nonumber\\
 =& -  \nu \sum_{e \in \mathcal{E}^{I}_h} \int_e \nabla \bu : [\bpsi_h] + \int_{\Omega_h} (\bff - \bff_h)\cdot \bpsi_h \\
 \le& C \nu h^\ell \|\bu\|_{H^\ell(\Omega)} \|\bpsi_h\|_{H^2_h(\Omega_h)}+ |{\bm f}-{\bm f}_h|_{X_h^*} \|\nab \bpsi_h\|_{L^2(\Omega_h)}\\
  \le& C \left(\nu h^\ell \|\bu\|_{H^\ell(\Omega)} + |{\bm f}-{\bm f}_h|_{X_h^*}\right)
  \|\bphi \|_{L^2(\Omega\cap \Omega_h)}.
\end{split}
\end{equation}
By an analogous argument, but with $s=2$ and $m=k$ in Lemma \ref{jumpBound} (so that $r=1$) , we have
\begin{equation}
\label{ibpPsi}
\begin{split}
a_h(\bu_h,\bpsi-\bpsi_h) 
%=& -\int_{\Omega_h} \bphi \cdot \bu_h + a_h(\bu_h,\bpsi) + \int_{\Omega_h} (\bphi - \bphi)\cdot \bu_h \nonumber\\
%=& \nu \int_{\Omega_h}\Delta \bpsi \cdot \bu_h -\nu \int_{\Omega_h}\Delta \bpsi \cdot \bu_h  -\nu \sum_{e \in \mathcal{E}^{I,\partial}} \int_e \nabla \bpsi : [\bu_h] \nonumber\\
=&  -\nu \sum_{e \in \mathcal{E}^{I}_h} \int_e \nabla \bpsi :[\bu_h] \\
\le &C \nu h^{k+1}\|\bpsi\|_{H^2(\Omega)} \|\bu_h\|_{H^k_h(\Omega_h)}\\
\le &C \nu h^{\ell} \|\bphi\|_{L^2(\Omega\cap \Omega_h)} \|\bu\|_{H^\ell(\Omega)}.
\end{split}
\end{equation}

Combining \eqref{start1}--\eqref{ibpPsi} yields
\begin{equation}
    a_h(\bu,\bpsi-\bpsi_h) + a_h(\bu-\bu_h, \bpsi_h)\le 
    C \left(\nu h^\ell \|\bu\|_{H^\ell(\Omega)} +  |\bff - \bff_h|_{\bX_h^*}\right) \|\bphi\|_{L^2(\Omega\cap \Omega_h)},
\end{equation}
and so applying this estimate and \eqref{amAh} to \eqref{breakUp} (recalling that $\bphi = \bu - \bu_h$ on $\Omega \cap \Omega_h$),
we obtain
\begin{align}\label{secondbound}
J_2\le C \left(\nu h^\ell \|\bu\|_{H^\ell(\Omega)} +  |\bff - \bff_h|_{\bX_h^*}\right).
\end{align}
Finally,  applying \eqref{firstBound} and \eqref{secondbound} to \eqref{splitUp}
completes the proof.
\end{proof}

%In particular, if we assume that $\bff$ is sufficiently smooth and $\bff_h$ is a nice enough interpolant of $\bff$ (such as a quadratic nodal interpolant), we have the bound
%\begin{equation*}
%\int_{\Omega_h)}(\bff - \bff_h)\cdot \bpsi_h \leq  C h^3 \|\bff\|_{H^3(\Omega)}\|\bphi\|_{L^2(\Omega_h \cap \Omega)}.
%\end{equation*}

\section{Numerical Experiments}\label{numerics}
We perform a series of numerical experiments to compare with the theoretical results presented in this paper. We focus on the $k=3$ case below. Numerical experiments for the $k=2$ case can be found in \cite{NeilanOtus21}.

We define our domain to be the region bounded by the ellipse 
\begin{equation*}
    \Omega = \{x\in \bbR^2:\ \frac{x_1^2}{2.25} + x_2^2 <1\},
\end{equation*}
and assign data according to the exact solution
\begin{equation}\label{testProblem}
\bu = \begin{pmatrix}
     1.5(\frac{x_1^2}{2.25} + x_2^2 - 1)(\frac{8x_1^2 y}{2.25} + \frac{x_1^2}{2.25} + 5x_2^2 - 1)\\
    \frac{-4x_1}{1.5}(\frac{x_1^2}{2.25} + x_2^2 - 1)(\frac{3x_1^2}{2.25} + x_2^2 + x_2 - 1)
\end{pmatrix}, \quad p = 10(\frac{x_1^2}{2.25} + x_2^2 - \frac{1}{2}).
\end{equation}

We take $\bff_h$ to be the cubic (nodal) Lagrange of $\bff$ and set the viscosity to $\nu = 1$ to compute the finite element method described in \eqref{eqn:FEM}. We subsequently compute the errors for decreasing mesh parameter $h$. 

\subsection{Isoparametric and affine comparison}
In Table \ref{isoVaffine}, we compare the isoparametric approximation defined through the Piola transform described in this paper with the corresponding affine approximation. Both tests were run on $\bpol^3 - \pol^2$ Scott-Vogelius elements with all edge degrees of freedom placed at the Gauss-Lobatto points. For the isoparametric approximation, we observe the optimal convergence rates predicted by the theory:
\begin{alignat*}{1}
\|\bu -\bu_h\|_{L^2(\Omega_h)} &= \mathcal{O}(h^4), \quad \|\nabla(\bu -\bu_h)\|_{L^2(\Omega_h)} = \mathcal{O}(h^3) \\
&\|p -p_h\|_{L^2(\Omega_h)} = \mathcal{O}(h^3).
\end{alignat*}
For the affine approximation, we observe suboptimal convergence.

\begin{table}[h]
\begin{tabular}{ | p{1cm} || p{2.5cm} |  p{1.5cm} ||p{2.5cm} | p{1.5cm} || p{2.5cm} | p{1.5cm} ||}
\hline
\multicolumn{7}{|c|}{Isoparametric} \\
\hline
$h$ & $\|\bu - \bu_h\|_{L^2(\Omega_h)}$ & rate & $\|\bu - \bu_h\|_{H^1(\Omega_h)}$ & rate & $\|p - p_h\|_{L^2(\Omega_h)}$ & rate \\
\hline
 0.654 & $3.391\cdot 10^{-1}$ & -- & $3.363$ & -- & $7.071$ & -- \\
\hline
 0.318 & $2.392\cdot 10^{-2}$ & 3.672 & $4.257\cdot 10^{-1}$ &2.862  & $5.234\cdot 10^{-1}$ & 3.604 \\
\hline
 0.158 & $1.675\cdot 10^{-3}$ & 3.791 & $6.232\cdot 10^{-2}$ & 2.740 & $8.845\cdot 10^{-2}$ & 2.537 \\
\hline
 0.079 & $1.139\cdot 10^{-4}$ & 3.866 & $9.046\cdot 10^{-3}$ & 2.776 & $1.298\cdot 10^{-2}$ & 2.761 \\
\hline
 0.039 & $7.183\cdot 10^{-6}$ & 3.985 & $1.225\cdot 10^{-3}$ & 2.882 & $1.695\cdot 10^{-3}$ & 2.935 \\
\hline 
\multicolumn{7}{|c|}{Affine} \\
\hline
$h$ & $\|\bu - \bu_h\|_{L^2(\Omega_h)}$ & rate & $\|\bu - \bu_h\|_{H^1(\Omega_h)}$ & rate & $\|p - p_h\|_{L^2(\Omega_h)}$ & rate \\
\hline
 0.654& $6.411\cdot 10^{-1}$   & -- & 3.705  & -- & 8.289  & -- \\
\hline
0.318 & $1.525\cdot 10^{-1}$   &  1.989  &  1.248  & 1.507    &  2.288  &  1.782  \\
\hline
 0.158 & $3.667\cdot 10^{-2}$  & 2.032   & $4.589\cdot 10^{-1}$   &  1.427  & $8.526\cdot 10^{-1}$  & 1.408  \\
\hline
0.079 & $8.779\cdot 10^{-3}$   &  2.056  & $1.635\cdot 10^{-1}$   & 1.484   & $3.005\cdot 10^{-1}$  & 1.500  \\
\hline
0.039 & $2.133\cdot 10^{-3}$   & 2.040   &  $5.741\cdot 10^{-2}$  & 1.509  & $1.056\cdot 10^{-1}$  & 1.508   \\
\hline   
\end{tabular}
\caption{Errors and rates for the Isoparametric approximations with Gauss-Lobatto nodes compared to the affine approximation.} %Here, $h$ refers to the largest distance between boundary vertices.} 
\label{isoVaffine}
\end{table}

\subsection{Dependence on degrees of freedom}
In Remark \ref{remEq}, we note %that the sharpness of our 
the error estimate may lose up to  $k-1$ powers of $h$ if equidistant nodes are used in places of Gauss-Lobatto nodes. To test this, we compute the errors for the isoparametric approximation with the standard, equidistant placement of degrees of freedom in order to test whether Gauss-Lobatto points are necessary or simply a tool for the
%convenient for 
analysis. We compare these results, shown in Table \ref{equiPoints}, with those in Table \ref{isoVaffine}, and we see that the isoparametric approximation with equidistant points is indeed suboptimal. % and matches the behavior predicted in Remark \ref{remEq}.  

\begin{table}[h]
\begin{tabular}{ | p{1cm} || p{2.5cm} |  p{1.5cm} ||p{2.5cm} | p{1.5cm} || p{2.5cm} | p{1.5cm} ||}
\hline
\multicolumn{7}{|c|}{Isoparametric with Equidistant Degrees of Freedom} \\
\hline
$h$ & $\|\bu - \bu_h\|_{L^2(\Omega_h)}$ & rate & $\|\bu - \bu_h\|_{H^1(\Omega_h)}$ & rate & $\|p - p_h\|_{L^2(\Omega_h)}$ & rate \\
\hline
 0.654 & $3.393\cdot 10^{-1}$  & -- & 3.298  & -- & 6.906  & -- \\
\hline
 0.318 & $2.307\cdot 10^{-2}$  & 3.723  & $3.998\cdot 10^{-1}$  & 2.922   & $4.813\cdot 10^{-1}$ &  3.689     \\
\hline
 0.158 & $1.657\cdot 10^{-3}$   & 3.755  & $5.538\cdot 10^{-2}$  & 2.819  & $8.089 \cdot 10^{-2}$  &  2.543   \\
\hline
 0.079 & $1.212\cdot 10^{-4}$ & 3.762  & $7.938\cdot 10^{-2}$  & 2.794  & $1.246 \cdot 10^{-2}$  &  2.691   \\
\hline
0.039 & $1.791\cdot 10^{-5}$   & 2.756  & $2.043\cdot 10^{-3}$  & 1.957  & $2.946 \cdot 10^{-3}$  & 2.079  \\
\hline 
\end{tabular}

\caption{Errors and rates for the Isoparametric approximation with degrees of freedom placed at standard, equidistant points.} \label{equiPoints}
\end{table}

\subsection{Divergence errors and pressure robustness}
We also compare the maximum divergence values computed using isoparametric approximation presented in this paper with those computed with the standard isoparametric approach. The degrees of freedom for both approximations are taken at the Gauss-Lobatto points so that the only difference is the use of the Piola transform in the velocity space. As shown in Figure \ref{plotDiv}, the method described in this paper is divergence free, whereas the standard isoparametric method ($u_h^{standard}$) is not.

\begin{figure}[h]
\includegraphics[width=0.75\textwidth, trim={0 7cm 0 6cm},clip]{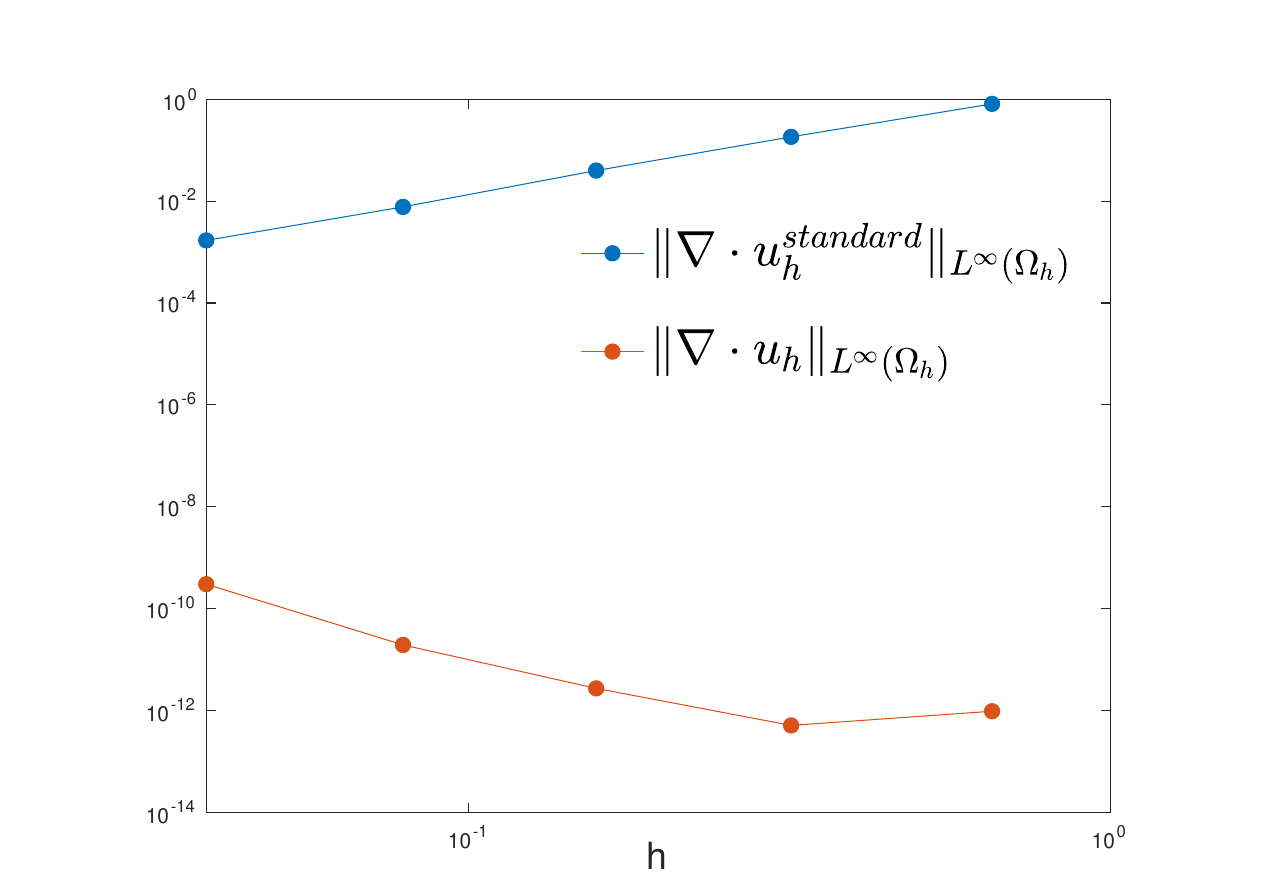}
\caption{Divergence of the isoparametric method with Piola transform compared to the standard isoparametric method on Scott-Vogelius $\bpol^3 - \pol^2$ elements.}\label{plotDiv}
\end{figure}

\begin{table}[h]
\begin{tabular}{ | p{1cm} || p{2.5cm} |  p{2.5cm} | }
\hline 
$\nu$ & $\|\bu - \bu_h \|_{L^2(\Omega_h)}$ & $\|\bu - \bu_h \|_{H^1(\Omega_h)}$ \\
\hline 
$10^{-7}$ & $1.139\cdot 10^{-4}$ & $9.046\cdot 10^{-3}$ \\
\hline 
$10^{-6}$ & $1.139\cdot 10^{-4}$ & $9.046\cdot 10^{-3}$ \\
\hline 
$10^{-3}$ & $1.139\cdot 10^{-4}$ & $9.046\cdot 10^{-3}$ \\
\hline 
1 & $1.139\cdot 10^{-4}$ & $9.046\cdot 10^{-3}$\\ 
\hline
\end{tabular}
\caption{Error tests for $h=0.079$ with varying values of viscosity $\nu$.}\label{robust}
\end{table}
Finally, we check the behavior of the method for varying values of viscosity. We run the method on $\bpol^3-\pol^2$ elements for data given by \eqref{testProblem}. In Table \ref{robust}, we show the behavior of the error in the velocity as we vary viscosity $\nu$. As we can see, the error remains nearly unchanged as we vary values of $\nu$ over several orders of magnitude, indicating that the scheme is pressure robust.

\bibliographystyle{siam}
\bibliography{ref}

\appendix
\section{Proof of Lemma \ref{lem:InvI}}\label{Ap:ProofInI}
\begin{proof}
Write $\bv(x) = A_T \hat \bv (\hat{x})$ 
%\textcolor{red}{(I added $x$ and $\hat{x}$ to avoid an abuse of notation--Becky)} 
for some $\hat \bv\in \hat \bV_k$.
We then use Lemma \ref{defAT}, \eqref{eqn:Bernardi},
and equivalence of norms
to obtain
\begin{equation}
\label{eqn:InvProof1}
\begin{split}
 \|\bv\|_{W^{\ell,p}(K)}
 &\le C h_T^{2/p-\ell} \|A_T \hat \bv\|_{W^{\ell,p}(\hat K)}\\
 &\le C h_T^{2/p-\ell} \|A_T\|_{W^{j,\infty}(\hat K)} \|\hat \bv\|_{W^{\ell,p}(\hat K)}\\
 &\le C h_T^{2/p-\ell-1}  \|\hat \bv\|_{W^{\ell,p}(\hat K)}\le 
 C h_T^{2/p-\ell-1}  \|\hat \bv\|_{L^{q}(\hat K)}.
\end{split}
\end{equation}
Likewise, we have
\begin{equation}
\label{eqn:InvProof2}
    \begin{split}
        \|\hat \bv\|_{L^q(\hat K)}\le \|A_T^{-1}\|_{L^\infty(\hat K)} \|A_T \hat \bv\|_{L^q(\hat K)}
        \le C h_T^{1-2/q} \|\bv\|_{L^q(K)}.
    \end{split}
\end{equation}
Combining \eqref{eqn:InvProof1}--\eqref{eqn:InvProof2}
yields \eqref{eqn:InvI1} for the case $m=0$.
The estimate \eqref{eqn:InvI1} for general $m$
then follows by standard arguments 
(cf.~\cite[Lemma 4.5.3]{brenner2008mathematical}).

 To prove \eqref{eqn:InvI2}, we first use \eqref{eqn:Bernardi}:
\begin{align*}
    |\bv|_{W^{\ell,p}(K)}
    &\le C  \Big[\underbrace{h_T^{2/p+\ell} \sum_{r=0}^k h_T^{-2r} |A_T \hat \bv|_{W^{r,p}(\hat K)}}_{=:I}    
    +\underbrace{h_T^{2/p+\ell} \sum_{r=k+1}^\ell h_T^{-2r} |A_T \hat \bv|_{W^{r,p}(\hat K)}}_{=:II}\Big].
\end{align*}
To bound $I$, we use \eqref{eqn:Bernardi} once again to obtain
% : \textcolor{red}{I still feel like this doesn't use Lemma \ref{defAT}. It seems to me that we are just bounding $\vert A_T \hat{v}\vert_{W^{r,p}(\hat{K})}$ with the second equation in \eqref{eqn:Bernardi}. What am I missing?}
\begin{align*}
    I\le h_T^{2/p+\ell}\sum_{r=0}^k h_T^{-2r} \cdot h_T^{r-2/p} \|\bv\|_{W^{r,p}(K)}\le C h_T^{\ell-k}\|\bv\|_{W^{k,p}(K)}.
\end{align*}
For $II$, we use the fact that $\hat \bv$ is a polynomial of degree $\le k$ on $K$ to obtain
\begin{align*}
    |A_T \hat \bv|_{W^{r,p}(\hat K)}
    &\le C \sum_{j=0}^k |A_T|_{W^{r-j,\infty}(\hat K)} |\hat \bv|_{W^{j,p}(\hat K)}\\
    &\le C\sum_{j=0}^k h_T^{r-j-1} |A_T^{-1} A_T \hat \bv|_{W^{j,p}(\hat K)}\\
    &\le C\sum_{j=0}^k \sum_{i=0}^j h_T^{r-j-1} |A_T^{-1}|_{W^{j-i,\infty}(\hat K)} |A_T \hat \bv|_{W^{i,p}(\hat K)}\\
    %&\le C\sum_{j=0}^k \sum_{i=0}^j h_T^{r-j-1} \cdot h_T^{j-i+1} |A_T \hat \bv|_{W^{i,p}(\hat K)}\\
    &\le C\sum_{j=0}^k \sum_{i=0}^j h_T^{r-i}  |A_T \hat \bv|_{W^{i,p}(\hat K)}\\
    &\le C\sum_{j=0}^k \sum_{i=0}^j h_T^{r-i}\cdot h_T^{i-2/p}  \|\bv\|_{W^{i,p}(K)}\\
    &\le C h_T^{r-2/p} \|\bv\|_{W^{k,p}(K)}.
\end{align*}
Thus,
\begin{align*}
    II \le C h_T^{2/p+\ell} \sum_{r=k+1}^\ell h_T^{-r-2/p} \|\bv\|_{W^{k,p}(K)}\le C \|\bv\|_{W^{k,p}(K)}.
    \end{align*}

Combining the bounds for $I$ and $II$ completes the proof of \eqref{eqn:InvI2}.
\end{proof}

\section{Proof of Lemma \ref{Eh}}\label{Ap:ProofEh}
\begin{proof}
    Define $\bE_h : \bV^h \to \bH_0^1(\Omega_h)$
    such that, for $\bv \in \bV^h$, 
    \begin{equation*}
        \bE_h \bv|_{T} = (\tilde{\bv} \circ F_{\tilde{T}}\circ F_{T}^{-1})|_{T},
    \end{equation*}
    where $\tilde{\bv}$ is the function in $\tilde{\bV}$ uniquely defined by 
    \begin{equation*}
        \bv|_T(a) = \tilde{\bv}|_{\tilde{T}}(\tilde{a}) \quad \forall{a} \in \mathcal{N}_T, \quad \forall T \in \mathcal{T}_h,
    \end{equation*}
    where $T= G_h(\tilde{T})$. In other words, in a \textit{standard} isoparametric, $k$th degree Lagrange 
    finite element method, $\bE_h \bv$ would be the function on the isoparametric element 
    associated with $\tilde{\bv}$ on $\tilde{T}$. Thus, $\bE_h\bv \in \bH_0^1(\Omega_h)$. 

    As shown in \cite{NeilanOtus21}, $\tilde{\bv} = \bE_h \bv$ on affine triangles, and we may conclude 
    \begin{equation*}
        \bE_h\bv|_T(a) = \bv|_T(a) \quad \forall{a} \in \mathcal{N}_T, \quad \forall T \in \mathcal{T}_h.
    \end{equation*}

    Our goal is to estimate $\bv - \bE_h \bv$, and our proof follows closely with the proof of Lemma 4.5 in \cite{NeilanOtus21}. However, here we provide a more general result. 
    
    As $\bv = \bE_h \bv$ on affine triangles, we only consider $T \in \mathcal{T}_h$ with curved boundaries. Additionally, we know $\bv |_{\partial T \cap \partial \Omega_h} = 0$. We may write $\bv|_T (x) = A_T (\hat{x})\hat{\bv}(\hat{x})$, for some $\hat{\bv}\in \hat{\bV}$, where $A_T = DF_T/\det{(DF_T)}$. Furthermore, there exists $\hat{\bw} \in \hat{\bV}$ such that $\hat{\bw}(\hat{x}) = \bE_h \bv|_{T}(x)$. Consequently,
    \begin{equation*}
        A_T(\hat{a})\hat{\bv}(\hat{a}) = \hat{\bw}(\hat{a}) \quad \forall \hat{a} \in \mathcal{N}_{\hat{T}},
    \end{equation*}
    so $\hat{\bw}$ is the piecewise $k$th degree Lagrange interpolant of $A_T\hat{\bv}$ on $\hat{T}^{CT}$.

    By the Bramble-Hilbert lemma, we have 
    \begin{equation}\label{eqn:BHL}
        \|A_T \hat{\bv} - \hat{\bw}\|_{H^i(\hat{K})} \leq C \vert A_T \hat{\bv}\vert_{H^{k+1}(\hat{K})} \quad \forall \hat{K} \in \hat{T}^{CT}, \quad i = 0,1,\ldots,k.
    \end{equation}

    We may then bound the right-hand side using Lemma \ref{defAT} and recognizing that $\hat{\bv}$ is a polynomial of degree $k$. Thus we have
\begin{equation}
    \label{startBound}
\begin{split}
     \vert A_T \hat{\bv}\vert_{H^{k+1}(\hat{K})} 
     \leq& 
     C \sum_{j=0}^{k+1}  \vert A_T\vert_{W^{k+1-j}(\hat K)} \vert \hat \bv \vert_{H^j(\hat K)} =  C \sum_{j=0}^{k}  \vert A_T\vert_{W^{k+1-j}(\hat K)} \vert \hat \bv \vert_{H^j(\hat K)}\\
     \le & C \sum_{j=0}^{k}  h_T^{k-j} \vert \hat \bv \vert_{H^j(\hat K)}.
     %
     % C\bigg( \|A_T D^3\hat{\bv}\|_{L^2(\hat{K})} + \|DA_T D^2 \hat{\bv}\|_{L^2(\hat{K})} + \|D^2 A_T D\hat{\bv}\|_{L^2(\hat{K})} + \|(D^3 A_T)\hat{\bv}\|_{L^2(\hat{K})} \bigg) \nonumber\\
%     \leq& C\bigg( \vert A_T \vert_{W^{1,\infty}(\hat{K})}\|D^2 \hat{\bv}\|_{L^2(\hat{K})} + \vert A_T \vert_{W^{2,\infty}(\hat{K})} \|D\hat{\bv}\|_{L^2(\hat{K})} + \vert A_T \vert_{W^{3,\infty}(\hat{K})} \|\bv \|_{L^2(\hat{K})}\bigg) \nonumber\\
%     \leq& C\bigg( \vert \hat{\bv} \vert_{H^2(\hat{K})} + h_T \vert \hat{\bv}\vert_{H^1(\hat{K})} + h_T^2 \|\hat{\bv}\|_{L^2(\hat{K})}\bigg). 
    \end{split}
    \end{equation}

Using Lemmas \ref{defAT} and \ref{scalingBernardi}, we have
\begin{equation}\label{second2Bound}
\begin{split}
    \vert \hat \bv \vert_{H^j(\hat K)} 
    &=     \vert  A_T^{-1} A_T \hat \bv \vert_{H^j(\hat K)}
    \le C \sum_{\ell=0}^j |A_T^{-1}|_{W^{j-\ell}(\hat K)} |A_T \hat \bv|_{H^\ell(\hat K)}\\
    &\le C \sum_{\ell=0}^j h_T^{1+j-\ell} |A_T \hat \bv|_{H^\ell(\hat K)}\\
    &\le C\sum_{\ell=0}^j h_T^{1+j-\ell} h_T^{\ell-1}\|\bv\|_{H^\ell( K)}\le C h_T^j \|\bv\|_{H^j(K)}.
\end{split}
\end{equation}
Inserting this estimate into \eqref{startBound} yields
\[
|A_T \hat \bv|_{H^{k+1}(\hat K)}\le C h_T^k \|\bv\|_{H^k(K)},
\]
and therefore by \eqref{eqn:BHL} and Lemma \ref{scalingBernardi},
\begin{align}\label{eqn:Ehk}
    |\bv-\bE_h \bv|_{H^i(K)}\le C h_T^{1-i} \|A_T \hat \bv-\hat \bw\|_{H^i(\hat K)}\le C h_T^{1-i} |A_T \hat \bv|_{H^{k+1}(\hat K)}\le C h_T^{k+1-i}\|\bv\|_{H^k(K)}.
\end{align}
An applicaiton of the inverse inequality \eqref{eqn:InvI1} then yields the desired
estimate \eqref{EhBound}. 
\end{proof}

\end{document}